\documentclass[sort&compress,3p]{elsarticle}
\usepackage{amsmath}
\usepackage{amssymb}
\usepackage{array}
\usepackage{array}
\usepackage{url}
\usepackage{graphicx}
\usepackage{tikz,ulem}
\usepackage{dsfont}
\usepackage{subfigure}
\usepackage{multirow}

\usepackage{algpseudocode,algorithm}

\newcommand{\K}{\mathcal{K}}
\newcommand{\R}{\mathbb{R}}
\newcommand{\I}{\mathbf{I}}
\newcommand{\V}{\mathbf{V}}
\newcommand{\J}{\mathbf{J}}
\newcommand{\Hb}{\mathbf{H}}
\newcommand{\A}{\mathbf{A}}

\newtheorem{theorem}{Theorem}
\newtheorem{remark}{Remark}
\newtheorem{lemma}{Lemma}
\newtheorem{definition}{Definition}
\newproof{proof}{Proof}

\begin{document}
\thispagestyle{empty}
\setcounter{page}{0}

\begin{Huge}
\begin{center}
Computational Science Laboratory Technical Report CSLTR-1/2014\\
\today
\end{center}
\end{Huge}
\vfil
\begin{huge}
\begin{center}
Paul Tranquilli and Adrian Sandu
\end{center}
\end{huge}

\vfil
\begin{huge}
\begin{it}
\begin{center}
``Exponential-Krylov methods for ordinary differential equations''
\end{center}
\end{it}
\end{huge}
\vfil
\textbf{Cite as:} Paul Tranquilli and Adrian Sandu.  Exponential-Krylov methods for ordinary differential equations. Journal of Computational Physics. Volume 278, Pages 31 -- 46, 2014.
\vfil

\begin{large}
\begin{center}
Computational Science Laboratory \\
Computer Science Department \\
Virginia Polytechnic Institute and State University \\
Blacksburg, VA 24060 \\
Phone: (540)-231-2193 \\
Fax: (540)-231-6075 \\ 
Email: \url{sandu@cs.vt.edu} \\
Web: \url{http://csl.cs.vt.edu}
\end{center}
\end{large}

\vspace*{1cm}

\begin{tabular}{ccc}
\includegraphics[width=2.5in]{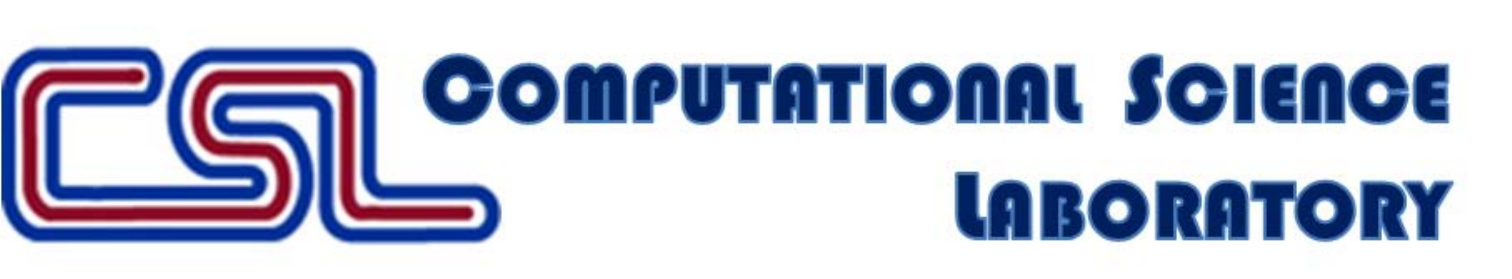}
&\hspace{2.5in}&
\includegraphics[width=2.5in]{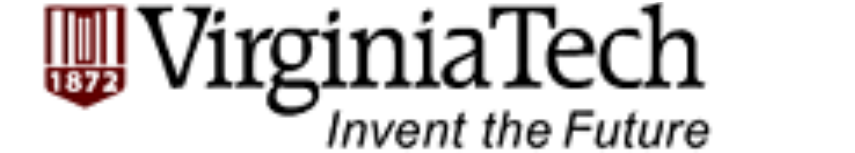} \\
{\bf\em Innovative Computational Solutions} &&\\
\end{tabular}

\newpage

\begin{frontmatter}
  \title{Exponential-Krylov methods for ordinary differential equations}
  \author[csl]{Paul Tranquilli}
  \ead{ptranq@vt.edu}
  \address[csl]{Computational Science Laboratory, Department of Computer Science, Virginia Tech.  Blacksburg, Virginia 24060}
  \author[csl]{Adrian Sandu}
  \ead{sandu@cs.vt.edu}

\begin{abstract}
This paper develops a new family of exponential time discretization methods called exponential-Krylov ({\sc expK}). 
The new schemes treat the time discretization and the Krylov based approximation of exponential matrix-vector products as a single computational process.  The classical order conditions theory developed herein accounts for both the temporal and the Krylov approximation errors. Unlike traditional exponential schemes, {\sc expK} methods require the construction of only a single Krylov space at each timestep. The number of basis vectors that guarantee the temporal order of accuracy does not depend on the application at hand.  Numerical results show favorable properties of {\sc expK} methods when compared to current exponential schemes.
 \end{abstract}

\end{frontmatter}

\section{Introduction}

Many methods exist to numerically approximate the solution of initial value problems 
\begin{equation}
\label{eqn:ode}
 \frac{dy}{dt} = f(t,y)\,,~~~ t_0 \leq t \leq t_F\,, ~~~ y(t_0) = y_n\,; \quad  y(t), f(t,y) \in \R^N\,.
\end{equation}
Multistep methods make use of the solution at several previous timesteps to compute the solution at $t_{n+1}$, while Runge-Kutta methods interpolate the solution at several points between the current solution, $t_n$, and the future solution, $t_{n+1}$.  In both cases implicit methods require the solution of (non-)linear system of equations at each time step.
Much work has been done towards the acceleration of the solutions to these systems. 
Iterative Krylov-based linear algebra solvers are the typical choice for large-scale applications \eqref{eqn:ode}. 
The generalized minimal residual (GMRES) method \cite{Saad} is the standard approach for constructing efficient solutions to linear systems arising throughout the integration of ODEs.  Jacobian-free Newton-Krylov (JFNK) methods \cite{Kelley_2004_Krylov,Keyes_2004_JFNK} make use of a GMRES like solver within a Newton iteration to solve the nonlinear equations arising from Runge-Kutta and multistep methods.  

Rosenbrock methods  \cite{Hairer_book_II}, a class of integrators coming from a linearization of Runge-Kutta methods, require only the solution of a linear system at each stage and are characterized by the explicit appearance of the Jacobian matrix 
\begin{equation*}
 \J_n = \left.\frac{\partial f}{\partial y}\right|_{t=t_n,y=y_n}
\end{equation*}
in the method itself.  Due to the approximate nature of solutions coming from iterative methods, the explicit appearance of $\J_n$ causes order reduction unless the system solution is very accurate.  For this reason Rosenbrock-W methods \cite{Hairer_book_II,Rang_2005_ROW3,Novati_2008_secantROW}, an extension of Rosenbrock methods allowing for arbitrary approximations of the matrix $\J_n$, have been developed.

Krylov-ROW methods \cite{Weiner_1998_order,Schmitt_1995_krylovW,Weiner_1997_rowmap,Podhaisky_1997_krylovW} couple Rosenbrock methods with Krylov based solvers for the linear systems arising therein.  A multiple Arnoldi process is used to enrich the Krylov space at each stage, and the order of the underlying Rosenbrock method is preserved with modest requirements on the Krylov space size, independent of the dimension of the ODE system under consideration.

The authors have recently developed Rosenbrock-K methods \cite{Tranquilli:2014} to pursue a similar goal. Krylov-ROW methods ensure the order results with standard Rosenbrock-W discretizations by adding requirements to the underlying Krylov space. In contradistinction, Rosenbrock-K methods guarantee the accuracy order through the use of a specific Krylov-based approximation of the Jacobian and the construction of new order conditions which take this approximation into account.  Rosenbrock-K methods have substantially fewer order conditions than Rosenbrock-W methods allowing for the construction of schemes of higher order with fewer stages.  More importantly, Rosenbrock-K methods give a strict lower bound on the number of Krylov basis vectors required for accuracy that depends only on the order of the method, and is completely independent of the dimension of the ODE system under consideration.

Exponential integrators \cite{Hochbruck:1998, Tokman_2011_EPIRK, Tokman_2006_EPI} replace the need to construct solutions to a linear system, or equivalently approximate the rational matrix function times vector product $\left( \I_N - h A\right)^{-1} v$, with the similar, hopefully cheaper, requirement to approximate the exponential matrix times vector product $\exp(h A) v$.  Like the solution of large linear systems, approximations of the matrix exponential times vectors are typically obtained using Krylov based methods.

In this paper we extend the ideas of the Rosenbrock-$K$ methods presented in \cite{Tranquilli:2014} to the particular set of exponential integrators discussed in Hochbruck, Lubich, and Selhofer \cite{Hochbruck:1998} 
and introduce the new family of exponential-Krylov (exponential-$K$) methods.  The new schemes require the construction of only a single Kyrlov basis at each timestep, as opposed to each stage in the case of standard exponential methods.  Moreover, the required dimension of the subspace to guarantee the desired order of accuracy is independent of the system \eqref{eqn:ode} under consideration.

The remainder of the paper is organized as follows. Section \ref{sec:formulation} presents the  exponential-$K$ framework and the Krylov approximation of the Jacobian used.  Section \ref{sec:orderconditions} develops the order condition theory for the new exponential-$K$ methods using both Butcher trees and $B$-series. Section \ref{sec:newmethods} constructs a practical four stage, fourth order exponential-$K$ method. Section \ref{sec:alternative} discusses alternative implementations of existing exponential methods, and Section \ref{sec:results} presents numerical results. Conclusions are drawn in Section \ref{sec:conclusions}.

\section{Formulation of exponential-Krylov methods}
\label{sec:formulation}

\subsection{Exponential-W integrators}
\label{sec:expW}
The starting point of our investigation is following class of exponential-W integrators proposed in \cite{Hochbruck:1998}
\begin{equation}
 \label{eqn:genformW}
  \begin{array}{rcl}
   k_i & = & \varphi(h\gamma \A_n)\left(h\, F_i + h\A_n\displaystyle\sum_{j=1}^{i-1}\gamma_{i,j}k_j\right), \\
   F_i & = & f\left( y_n +  \displaystyle\sum_{j=1}^{i-1} \alpha_{i,j}k_j \right), \\
   y_{n+1} & = & y_n +  \displaystyle\sum_{j=1}^s b_i k_i,
  \end{array}
\end{equation}
where $\A_n$ is either the matrix $\J_n$ or an approximation of it. Equation (\ref{eqn:genformW}) formalizes explicit Runge-Kutta methods when $\varphi(z) = 1$, Rosenbrock methods when $\varphi(z) = 1/(1-z)$, and exponential methods when $\varphi(z) = (e^z-1)/z$.  Note that similar to the Rosenbrock methods discussed before, equation (\ref{eqn:genformW}) makes explicit use of the matrix $\J_n$, and so it is natural to explore conditions allowing for arbitrary approximations as in the case of Rosenbrock-W methods.  A discussion of these methods and their order conditions is given in \cite{Hochbruck:1998}.

\subsection{Exponential-K integrators}
\label{sec:expK}

The new exponential-K methods proposed in this work have the same general form as exponential-W methods  (\ref{eqn:genformW}), but use a  specific, Krylov based-approximation $\A_n$ of the Jacobian. To begin we construct the $M$-dimensional Krylov space $\K_M$ where $M \ll N$ and
\begin{equation}
\begin{array}{rcl}
\label{eqn:krylovspace}
 \K_M & = & \textrm{span}\left\{f_n, \J_n f_n, \J_n^2f_n, \dots, \J_n^{M-1}f_n\right\} \\
          & = & \textrm{span}\left\{v_1, v_2, \dots, v_M\right\}
\end{array}
\end{equation}
using a modified Arnoldi iteration \cite{vanderVorst}.  The Arnoldi iteration returns the matrix
\begin{equation*}
\V = \left[v_1, v_2, \dots, v_M\right] \in \R^{N\times M}
\end{equation*}
whose columns form an orthonormal basis of $\K_M$, and the upper Hessenberg matrix
\begin{equation}
\label{eqn:arnH}
\Hb = \V^T\,\J_n\,\V \in \R^{M \times M}.
\end{equation}
From these two matrices we construct the following Krylov-based approximation of the Jacobian
\begin{equation}
\label{eqn:approxJac}
 \A_n = \V\,\Hb\,\V^T = \V\,\V^T\,\J_n\,\V\,\V^T.
\end{equation}
The powers of this matrix have the following property.
\begin{lemma}[Powers of $\A_n$]
\label{lem:powers}
For any $k = 0, 1, 2, \dots$
 \[
 \A_n^k = \V\,\Hb^k\,\V^T.
 \]
\end{lemma}
\begin{proof}
 We give the proof of the Lemma by induction.  As the base case we have that
 \[\A_n^1 = \V\,\Hb\,\V^T = \V\,\Hb^1\,\V^T\]
 next we assume that $\A_n^{k-1} = \V\,\Hb^{k-1}\,\V^T$ and show that $\A_n^k = \V\,\Hb^k\,\V^T$.
 \[\A_n^k = \A_n\A_n^{k-1} = \V\,\Hb\,\V^T\left(\V\,\Hb^{k-1}\,\V^T\right) = \V\,\Hb\, \left( \V^T\,\V\right)\Hb^{k-1}\,\V^T\]
 because $V$ is an orthonormal matrix $\V^T\,\V = \I_N$, and so
 \[\A_n^k = \V\,\Hb^k\,\V^T\]
 \qed
\end{proof}

The construction of exponential integrators uses matrix functions of the form $\varphi_k(h\gamma \A_n)$,
where the functions are defined by 
\begin{equation*}
\varphi_k(z) = \displaystyle\int_0^1 e^{z(1-\theta)}\frac{\theta^{k-1}}{(k-1)!} \, d\theta, \quad k = 0, 1, 2, \dots
\end{equation*}
and satisfy the recurrence relation
\begin{equation*}
\varphi_{k+1}(z) = \frac{\varphi_k(z) - 1/k!}{z}, \quad \varphi_k(0) = \frac{1}{k!}.
\end{equation*}
The matrix functions have the following property.

\begin{lemma}[Matrix functions of the approximate Jacobian]
\label{lem:redphi}
 \[
 \varphi_k(h\gamma \A_n) = \frac{1}{k!}(\I_N - \V\,\V^T) + \V\,\varphi_k(h\gamma \Hb)\,\V^T
\quad k = 1, 2, \dots
 \]
\end{lemma}
\begin{proof}
It is possible to expand $\varphi_k(z)$ as a Taylor series \cite{Loffeld_2013_Comparison}:
\begin{equation}
 \label{eqn:varphitaylor}
 \varphi_k(z) = \displaystyle\sum_{i = 0}^\infty c_i \frac{z^i}{(k+i)!}
\end{equation}
We have that
\[ \varphi_k(h\gamma \A_n) = \frac{1}{k!}\I_N + \displaystyle\sum_{i=1}^{\infty} \frac{(h\gamma)^i}{(k+i)!}\A_n^i \]
after applying Lemma \ref{lem:powers} we obtain
\[ \varphi_k(h\gamma \A_n) = \frac{1}{k!}\I_N + \displaystyle\sum_{i=1}^{\infty} \frac{(h\gamma)^i}{(k+i)!}\V\Hb^i\V^T .\]
Similarly we can expand $\V\varphi_k(h\gamma H)\V^T$ as
\[ \V\varphi_k(h\gamma \Hb)\V^T = \frac{1}{k!}\V\V^T + \displaystyle\sum_{i=1}^{\infty} \frac{(h\gamma)^i}{(k+i)!}\V\Hb^i\V^T, \]
taking the difference we see that
\[ \varphi_k(h\gamma \A_n) - \V\varphi_k(h\gamma \Hb)\V^T = \frac{1}{k!}\left(\I_N - \V\V^T\right). \]
Finally we move $\V\varphi_k(h\gamma H)\V^T$ across the equality to obtain
\begin{equation}
 \label{eqn:reducedphi}
 \varphi_k(h\gamma \A_n) = \frac{1}{k!} \left(\I_N - \V\V^T\right) + \V\varphi_k(h\gamma \Hb)\V^T
\end{equation}
\qed
\end{proof}

To finish the derivation of a reduced form for the exponential-Krylov integrator (\ref{eqn:genformW}) we introduce the following notation
\begin{equation*}
 k_i = \underbrace{\V \lambda_i}_{\in \K_M} + \underbrace{\mu_i}_{\in \K_M^\perp}, \quad F_i = \underbrace{\V\psi_i}_{\in \K_M} + \underbrace{\delta_i}_{\in \K_M^\perp}
\end{equation*}
where $\V\lambda_i$ and $\V\psi_i$ represent the components of $k_i$ and $F_i$ which reside in the Krylov subspace $\K_M$, and similarly $\mu_i$ and $\delta_i$ are the components residing in the space orthogonal to $\K_M$.  Insert equation (\ref{eqn:reducedphi}) and the split forms of $k_i$ and $F_i$ into the general method formulation (\ref{eqn:genformW}) to obtain
\begin{equation*}
\V\,\lambda_i + \mu_i = \V\left(\varphi(h\gamma \Hb)\psi_i + h \Hb \displaystyle\sum_{j=1}^{i-1} \gamma_{i,j}\lambda_j\right) + \delta_i\,.
\end{equation*}
This leads to the following equation for the reduced stage vector
\begin{equation}
\label{eqn:redstage}
\lambda_i = \varphi(h\gamma \Hb)\left(\psi_i + h\Hb\displaystyle\sum_{j=1}^{i-1}\gamma_{i,j}\lambda_j\right)\,.
\end{equation}
The full stage values can be recovered as 
\begin{equation}
k_i = \V\,\lambda_i + (F_i - \V\psi_i).
\end{equation}
A single step of an autonomous exponential-K method is given in Algorithm \ref{alg:expK-autonomous-step}.
\begin{algorithm}[ht]
\caption{One step of an autonomous exponential-K integrator}\label{alg:expK-autonomous-step}
\begin{algorithmic}[1]
   \State Compute $\Hb$ and $\V$ using the $N$-dimensional Arnoldi process \cite{vanderVorst}
   \For{$i = 1,\dots,s$}\Comment{For each stage, in succession}
      \begin{eqnarray*}
       F_i &=& f\left( y_n + \displaystyle\sum_{j=1}^{i-1}\alpha_{i,j}k_j\right) \\
       \psi_i &=& \V^T\, F_i \\
      \lambda_i &=& \varphi\left(h\gamma \Hb\right)\, \left( h\psi_i + h \Hb \displaystyle\sum_{j=1}^{i-1} \gamma_{i,j} \lambda_j\right) \\
       k_i &=& \V\, \lambda_i + h\, (F_i - \V\, \psi_i)
     \end{eqnarray*}
   \EndFor
   \State $y_{n+1} = y_n + \displaystyle\sum_{i=1}^s b_i k_i$
\end{algorithmic}
\end{algorithm}
%

%
\begin{remark}
Because the matrix $\Hb$ has dimension $M \times M$ direct methods can be used to compute the matrix function $\varphi(h\gamma\Hb)$.  In the case of Rosenbrock methods where $\varphi(z) = 1/(1-z)$ a direct $LU$-decomposition can be used. For exponential methods where $\varphi(z) = (e^z-1)/z$ a Pade approximation \cite{Higham:2005} can be utilized.  Furthermore, a single matrix function needs to be evaluated at each step when the matrices $\Hb$ and $\V$ are constructed.
\end{remark}
\begin{remark}
We have only given here the autonomous form of an exponential-$K$ method.  A non-autonomous form is possible by constructing an extended ODE system and the corresponding Jacobian.  
This construction follows closely the treatment given in \cite{Tranquilli:2014}, where an $N+1$ dimensional Arnoldi iteration is discussed.
\end{remark}

\section{Order conditions for exponential-K methods}
\label{sec:orderconditions}
We construct classical order conditions for exponential-K methods.  To this end we match the Taylor series expansion of the numerical and exact solutions up to a specified order.  Butcher-trees\cite{Hairer_book_I} are an established method of representing terms in the Taylor series expansions of Runge-Kutta like methods.  The derivation of order conditions for $K$-methods is an extension of the framework developed for $W$-methods.  The theory for $W$ methods is constructed using $TW$-trees, a subclass of $P$-trees, which are themselves an extension of $T$-trees that allow for two different colored nodes.
\[ 
TW = \left\{ \begin{array}{cl}P\textrm{-trees:} & \textrm{end vertices are meagre, and} \\
				 & \textrm{fat vertices are singly branched} \end{array} \right\} 
\]

In the context of $TW$(and $TK$)-trees a meagre, or solid, node represents an appearance of the exact Jacobian matrix $\J_n$, while a fat, or empty, node represents the appearance of the approximate Jacobian matrix $\A_n$.  Each tree represents a single elementary differential in the Taylor series of either the exact or numerical solutions of the ODE.  

\begin{figure}[ht]
\begin{center}
\def\arraystretch{1.5}
\footnotesize
\begin{tabular}{|c|c|c|c|c|c|}
\hline
$\tau$ & \begin{tikzpicture}[scale=.5]
      [meagre/.style={circle,draw, fill=black!100,thick},
      fat/.style={circle,draw,thick}]
      \node[circle,draw,, fill=black!100,thick] (j) at (0,0) [label=right:$j$] {};
  \end{tikzpicture} & \begin{tikzpicture}[scale=.5]
      [meagre/.style={circle,draw, fill=black!100,thick},
      fat/.style={circle,draw,thick}]
      \node[circle,draw, fill=black!100,thick] (j) at (0,0) [label=right:$j$] {};
      \node[circle,draw, fill=black!100,thick] (k) at (1,1) [label=right:$k$] {};
      \draw[-] (j) -- (k);
  \end{tikzpicture} & \begin{tikzpicture}[scale=.5]
      [meagre/.style={circle,draw, fill=black!100,thick},
      fat/.style={circle,draw,thick}]
      \node[circle,draw, thick] (j) at (0,0) [label=right:$j$] {};
      \node[circle,draw, fill=black!100,thick] (k) at (1,1) [label=right:$k$] {};
      \draw[-] (j) -- (k);
  \end{tikzpicture} & \begin{tikzpicture}[scale=.5]
      [meagre/.style={circle,draw, fill=black!100,thick},
      fat/.style={circle,draw,thick}]
      \node[circle,draw, fill=black!100,thick] (j) at (0,0) [label=right:$j$] {};
      \node[circle,draw, fill=black!100,thick] (k) at (1,1) [label=right:$k$] {};
      \node[circle,draw, fill=black!100,thick] (l) at (-1,1) [label=left:$l$] {};
      \draw[-] (j) -- (k);
      \draw[-] (j) -- (l);
  \end{tikzpicture} & \begin{tikzpicture}[scale=.5]
      [meagre/.style={circle,draw, fill=black!100,thick},
      fat/.style={circle,draw,thick}]
      \node[circle,draw, fill=black!100,thick] (j) at (0,0) [label=right:$j$] {};
      \node[circle,draw, fill=black!100,thick] (k) at (1,1) [label=right:$k$] {};
      \node[circle,draw, fill=black!100,thick] (l) at (0,2) [label=left:$l$] {};
      \draw[-] (j) -- (k);
      \draw[-] (k) -- (l);
  \end{tikzpicture}   \\
\hline
$F(\tau)$ & $f^J$ & $f^J_Kf^K$ & $\A_{JK}f^K$ & $f^J_{KL}f^Kf^L$ & $f^J_Kf^K_Lf^L$   \\
\hline
$\mathsf{a}(\tau)$ & $x_1$ & $x_2$ & $x_3$ & $x_4$ & $x_5$ \\
\hline
$B^\#\left(hf(B(\mathsf{a},y))\right)$ & 1 & $x_1$ & $0$ & $x_1^2$ & $x_2$ \\
\hline
$B^\#\left(h\A B(\mathsf{a},y)\right)$ & $0$ & $0$ & $x_1$ & $0$ & $0$ \\
\hline
$B^\#\left(\varphi(h\gamma \A)B(\mathsf{a},y)\right)$ & $x_1$ & $x_2$ & $x_3 + c_1x_1$ & $x_4$ & $x_5$ \\
\hline
\hline
$\tau$ & \begin{tikzpicture}[scale=.5]
      [meagre/.style={circle,draw, fill=black!100,thick},
      fat/.style={circle,draw,thick}]
      \node[circle,draw, fill=black!100,thick] (j) at (0,0) [label=right:$j$] {};
      \node[circle,draw, thick] (k) at (1,1) [label=right:$k$] {};
      \node[circle,draw, fill=black!100,thick] (l) at (0,2) [label=left:$l$] {};
      \draw[-] (j) -- (k);
      \draw[-] (k) -- (l);
  \end{tikzpicture} & \begin{tikzpicture}[scale=.5]
      [meagre/.style={circle,draw, fill=black!100,thick},
      fat/.style={circle,draw,thick}]
      \node[circle,draw,thick] (j) at (0,0) [label=right:$j$] {};
      \node[circle,draw, fill=black!100, thick] (k) at (1,1) [label=right:$k$] {};
      \node[circle,draw, fill=black!100,thick] (l) at (0,2) [label=left:$l$] {};
      \draw[-] (j) -- (k);
      \draw[-] (k) -- (l);
  \end{tikzpicture} & \begin{tikzpicture}[scale=.5]
      [meagre/.style={circle,draw, fill=black!100,thick},
      fat/.style={circle,draw,thick}]
      \node[circle,draw, thick] (j) at (0,0) [label=right:$j$] {};
      \node[circle,draw, thick] (k) at (1,1) [label=right:$k$] {};
      \node[circle,draw, fill=black!100,thick] (l) at (0,2) [label=left:$l$] {};
      \draw[-] (j) -- (k);
      \draw[-] (k) -- (l);
  \end{tikzpicture} & \begin{tikzpicture}[scale=.5]
      [meagre/.style={circle,draw, fill=black!100,thick},
      fat/.style={circle,draw,thick}]
      \node[circle,draw, fill=black!100,thick] (j) at (0,0) [label=right:$j$] {};
      \node[circle,draw, fill=black!100,thick] (k) at (-1,1) [label=left:$k$] {};
      \node[circle,draw, fill=black!100,thick] (l) at (0,1) [label=above:$l$] {};
      \node[circle,draw, fill=black!100,thick] (m) at (1,1) [label=right:$m$] {};
      \draw[-] (j) -- (k);
      \draw[-] (j) -- (l);
      \draw[-] (j) -- (m);
  \end{tikzpicture} & \begin{tikzpicture}[scale=.5]
      [meagre/.style={circle,draw, fill=black!100,thick},
      fat/.style={circle,draw,thick}]
      \node[circle,draw, fill=black!100,thick] (j) at (0,0) [label=right:$j$] {};
      \node[circle,draw, fill=black!100,thick] (k) at (-1,1) [label=left:$k$] {};
      \node[circle,draw, fill=black!100,thick] (l) at (1,1) [label=right:$l$] {};
      \node[circle,draw, fill=black!100,thick] (m) at (1,2) [label=right:$m$] {};
      \draw[-] (j) -- (k);
      \draw[-] (j) -- (l);
      \draw[-] (l) -- (m);
  \end{tikzpicture}  \\
\hline
$F(\tau)$ & $f^J_K\A_{KL}f^L$ & $\A_{JK}f^K_Lf^L$ & $\A_{JK}\A_{KL}f^L$ & $f^J_{KLM}f^Kf^Lf^M$ & $f^J_{KL}f^L_Mf^Mf^K$ \\
\hline
$\mathsf{a}(\tau)$ & $x_6$ & $x_7$ & $x_8$ & $x_9$ & $x_{10}$ \\
\hline
$B^\#\left(hf(B(\mathsf{a},y))\right)$& $x_3$ & $0$ & $0$ & $x_1^3$ & $x_1x_2$ \\
\hline
$B^\#\left(h\A B(\mathsf{a},y)\right)$ & $0$ & $x_2$ & $x_3$ & $0$ & $0$ \\
\hline
$B^\#\left(\varphi(h\gamma \A)B(\mathsf{a},y)\right)$ & $x_6$ & $x_7 + c_1x_2$ & $x_8 + c_1x_3 + c_2x_1$ & $x_9$ & $x_{10}$ \\
\hline
\hline
$\tau$  &  \begin{tikzpicture}[scale=.5]
      [meagre/.style={circle,draw, fill=black!100,thick},
      fat/.style={circle,draw,thick}]
      \node[circle,draw, fill=black!100,thick] (j) at (0,0) [label=right:$j$] {};
      \node[circle,draw, fill=black!100,thick] (k) at (-1,1) [label=left:$k$] {};
      \node[circle,draw, thick] (l) at (1,1) [label=right:$l$] {};
      \node[circle,draw, fill=black!100,thick] (m) at (1,2) [label=right:$m$] {};
      \draw[-] (j) -- (k);
      \draw[-] (j) -- (l);
      \draw[-] (l) -- (m);
  \end{tikzpicture} & \begin{tikzpicture}[scale=.5]
      [meagre/.style={circle,draw, fill=black!100,thick},
      fat/.style={circle,draw,thick}]
      \node[circle,draw, fill=black!100,thick] (j) at (0,0) [label=right:$j$] {};
      \node[circle,draw, fill=black!100,thick] (k) at (1,1) [label=right:$k$] {};
      \node[circle,draw, fill=black!100,thick] (l) at (2,2) [label=right:$l$] {};
      \node[circle,draw, fill=black!100,thick] (m) at (0,2) [label=left:$m$] {};
      \draw[-] (j) -- (k);
      \draw[-] (k) -- (l);
      \draw[-] (k) -- (m);
  \end{tikzpicture} & \begin{tikzpicture}[scale=.5]
      [meagre/.style={circle,draw, fill=black!100,thick},
      fat/.style={circle,draw,thick}]
      \node[circle,draw,thick] (j) at (0,0) [label=right:$j$] {};
      \node[circle,draw, fill=black!100,thick] (k) at (1,1) [label=right:$k$] {};
      \node[circle,draw, fill=black!100,thick] (l) at (2,2) [label=right:$l$] {};
      \node[circle,draw, fill=black!100,thick] (m) at (0,2) [label=left:$m$] {};
      \draw[-] (j) -- (k);
      \draw[-] (k) -- (l);
      \draw[-] (k) -- (m);
  \end{tikzpicture} & \begin{tikzpicture}[scale=.5]
      [meagre/.style={circle,draw, fill=black!100,thick},
      fat/.style={circle,draw,thick}]
      \node[circle,draw, fill=black!100,thick] (j) at (0,0) [label=left:$j$] {};
      \node[circle,draw, fill=black!100,thick] (k) at (1,1) [label=right:$k$] {};
      \node[circle,draw, fill=black!100,thick] (l) at (0,2) [label=left:$l$] {};
      \node[circle,draw, fill=black!100,thick] (m) at (1,3) [label=right:$m$] {};
      \draw[-] (j) -- (k);
      \draw[-] (k) -- (l);
      \draw[-] (l) -- (m);
  \end{tikzpicture} & \begin{tikzpicture}[scale=.5]
      [meagre/.style={circle,draw, fill=black!100,thick},
      fat/.style={circle,draw,thick}]
      \node[circle,draw,fill=black!100,thick] (j) at (0,0) [label=left:$j$] {};
      \node[circle,draw,fill=black!100,thick] (k) at (1,1) [label=right:$k$] {};
      \node[circle,draw, thick] (l) at (0,2) [label=left:$l$] {};
      \node[circle,draw, fill=black!100,thick] (m) at (1,3) [label=right:$m$] {};
      \draw[-] (j) -- (k);
      \draw[-] (k) -- (l);
      \draw[-] (l) -- (m);
  \end{tikzpicture} \\
\hline
$F(\tau)$ & $f^J_{KL}\A_{LM}f^Mf^K$ & $f^J_Kf^K_{LM}f^Mf^L$ & $\A_{JK}f^K_{LM}f^Lf^M$ & $f^J_Kf^K_Lf^L_Mf^M$ & $ f^J_Kf^K_L\A_{LM}f^M $ \\
\hline
$\mathsf{a}(\tau)$ & $x_{11}$ & $x_{12}$ & $x_{13}$ & $x_{14}$ & $x_{15}$ \\
\hline
$hf(B(\mathsf{a},y))$ & $x_1x_3$ & $x_4$ & $0$ & $x_5$ & $x_6$ \\
\hline
$B^\#\left(h\A B(\mathsf{a},y)\right)$ & $0$ & $0$ & $x_4$ & $0$ & $0$ \\
\hline
$B^\#\left(\varphi(h\gamma \A)B(\mathsf{a},y)\right)$ & $x_{11}$ & $x_{12}$ & $x_{13} + c_1 x_4$ & $x_{14}$ & $x_{15}$ \\
\hline 
\end{tabular}
\caption{TW-trees up to order four (part one of two).}
\label{fig:TWtrees}
\end{center}
\end{figure}
\begin{figure}[ht]
\begin{center}
\renewcommand{\arraystretch}{1.5}
\footnotesize
\begin{tabular}{|c|c|c|c|}
\hline
$\tau$ & \begin{tikzpicture}[scale=.5]
      [meagre/.style={circle,draw, fill=black!100,thick},
      fat/.style={circle,draw,thick}]
      \node[circle,draw, fill=black!100,thick] (j) at (0,0) [label=left:$j$] {};
      \node[circle,draw, thick] (k) at (1,1) [label=right:$k$] {};
      \node[circle,draw, fill=black!100,thick] (l) at (0,2) [label=left:$l$] {};
      \node[circle,draw, fill=black!100,thick] (m) at (1,3) [label=right:$m$] {};
      \draw[-] (j) -- (k);
      \draw[-] (k) -- (l);
      \draw[-] (l) -- (m);
  \end{tikzpicture} & \begin{tikzpicture}[scale=.5]
      [meagre/.style={circle,draw, fill=black!100,thick},
      fat/.style={circle,draw,thick}]
      \node[circle,draw,fill=black!100, thick] (j) at (0,0) [label=left:$j$] {};
      \node[circle,draw, thick] (k) at (1,1) [label=right:$k$] {};
      \node[circle,draw, thick] (l) at (0,2) [label=left:$l$] {};
      \node[circle,draw, fill=black!100,thick] (m) at (1,3) [label=right:$m$] {};
      \draw[-] (j) -- (k);
      \draw[-] (k) -- (l);
      \draw[-] (l) -- (m);
  \end{tikzpicture} & \begin{tikzpicture}[scale=.5]
      [meagre/.style={circle,draw, fill=black!100,thick},
      fat/.style={circle,draw,thick}]
      \node[circle,draw, thick] (j) at (0,0) [label=left:$j$] {};
      \node[circle,draw, fill=black!100, thick] (k) at (1,1) [label=right:$k$] {};
      \node[circle,draw, fill=black!100,thick] (l) at (0,2) [label=left:$l$] {};
      \node[circle,draw, fill=black!100,thick] (m) at (1,3) [label=right:$m$] {};
      \draw[-] (j) -- (k);
      \draw[-] (k) -- (l);
      \draw[-] (l) -- (m);
  \end{tikzpicture} \\
\hline 
$F(\tau)$ & $f^J_K\A_{KL}f^L_Mf^M$ & $f^J_K\A_{KL}\A_{LM}f^M$ & $\A_{JK}f^K_Lf^L_Mf^M$ \\
\hline 
$\mathsf{a}(\tau)$ & $x_{16}$ & $x_{17}$ & $x_{18}$ \\
\hline
$B^\#\left(hf(B(\mathsf{a},y))\right)$ & $x_7$ & $x_8$ & $0$ \\
\hline
$B^\#\left(h\A B(\mathsf{a},y)\right)$ & $0$ & $0$ & $x_5$ \\
\hline
$B^\#\left(\varphi(h\gamma \A)B(\mathsf{a},y)\right)$ & $x_{16}$ & $x_{17}$ & $x_{18} + c_1x_5$ \\
\hline
\hline
$\tau$ &  \begin{tikzpicture}[scale=.5]
      [meagre/.style={circle,draw, fill=black!100,thick},
      fat/.style={circle,draw,thick}]
      \node[circle,draw, thick] (j) at (0,0) [label=left:$j$] {};
      \node[circle,draw, fill=black!100,thick] (k) at (1,1) [label=right:$k$] {};
      \node[circle,draw, thick] (l) at (0,2) [label=left:$l$] {};
      \node[circle,draw, fill=black!100,thick] (m) at (1,3) [label=right:$m$] {};
      \draw[-] (j) -- (k);
      \draw[-] (k) -- (l);
      \draw[-] (l) -- (m);
  \end{tikzpicture} & \begin{tikzpicture}[scale=.5]
      [meagre/.style={circle,draw, fill=black!100,thick},
      fat/.style={circle,draw,thick}]
      \node[circle,draw, thick] (j) at (0,0) [label=left:$j$] {};
      \node[circle,draw, thick] (k) at (1,1) [label=right:$k$] {};
      \node[circle,draw, fill=black!100,thick] (l) at (0,2) [label=left:$l$] {};
      \node[circle,draw, fill=black!100,thick] (m) at (1,3) [label=right:$m$] {};
      \draw[-] (j) -- (k);
      \draw[-] (k) -- (l);
      \draw[-] (l) -- (m);
  \end{tikzpicture} & \begin{tikzpicture}[scale=.5]
      [meagre/.style={circle,draw, fill=black!100,thick},
      fat/.style={circle,draw,thick}]
      \node[circle,draw, thick] (j) at (0,0) [label=left:$j$] {};
      \node[circle,draw, thick] (k) at (1,1) [label=right:$k$] {};
      \node[circle,draw, thick] (l) at (0,2) [label=left:$l$] {};
      \node[circle,draw, fill=black!100,thick] (m) at (1,3) [label=right:$m$] {};
      \draw[-] (j) -- (k);
      \draw[-] (k) -- (l);
      \draw[-] (l) -- (m);
  \end{tikzpicture} \\
\hline
$F(\tau)$ & $ \A_{JK}f^K_L\A_{LM}f^M$ & $ \A_{JK}\A_{KL}f^L_Mf^M$ & $\A_{JK}\A_{KL}\A_{LM}f^M$ \\
\hline
$\mathsf{a}(\tau)$ &  $x_{19}$ & $x_{20}$ & $x_{21}$ \\
\hline 
$B^\#\left(hf(B(\mathsf{a},y))\right)$ &$0$ & $0$ & $0$ \\
\hline 
$B^\#\left(h\A B(\mathsf{a},y)\right)$ & $x_6$ & $x_7$ & $x_8$ \\
\hline 
$B^\#\left(\varphi(h\gamma \A)B(\mathsf{a},y)\right)$ & $x_{19} + c_1x_6$ & $x_{20} + c_1x_7 + c_2x_2$ & $x_{21} + c_1x_8 + c_2x_3 + c_3x_1$ \\
\hline 
\end{tabular}
\caption{TW-trees up to order four (part two of two).}
\label{fig:TWtrees2}
\end{center}
\end{figure}

A fundamental component of our derivation of order conditions for exponential-$K$ methods are $B$-series, a way of representing an expansion in trees, or elementary differentials, as a sequence of real numbers.  A mapping $\mathsf{a} : TW \cup \left\{\emptyset\right\} \rightarrow \R$ represents the series
\begin{equation*}
B(\mathsf{a},y) = \mathsf{a(}\emptyset)y + \displaystyle\sum_{\tau \in TW} a(\tau) \frac{h^{\left| \tau \right|}}{\sigma(\tau)}F(\tau)(y).
\end{equation*}
Here $\tau$ are TW-trees; the order $\left|\tau\right|$, and the symmetry $\sigma(\tau)$ of a tree are defined in the same way as for single colored trees \cite{Hairer_book_I, Tokman_2011_EPIRK}, and $F(\tau)$ is the elementary differential belonging to tree $\tau$ as in figure \ref{fig:TWtrees}.  Similarly we introduce the operator $B^\#(f)$, which takes as input a function (which can be represented by a series) and returns the B-series coefficients
\begin{equation*}
 B^\#\left(B(\mathsf{a},y)\right) = \mathsf{a}.
\end{equation*}

Figure \ref{fig:TWtrees} shows all $TW$-trees to order four, the coefficients of a generic $B$-series $B(\mathsf{a},y)$, the result of composing $B(\mathsf{a},y)$ with the function $f(y)$, the result of a multiplication of $B(\mathsf{a},y)$ with the Krylov approximation matrix $\A_n$, and the result of a multiplication of $B(\mathsf{a},y)$ by $\varphi(h\gamma \A_n)$.  The full details of the composition of B-series can be found in \cite{Chartier_2010_Bseries}, while details of products with the Jacobian and $\varphi$-functions can be found in \cite{Butcher_2009_Bseries}.

Because $K$-methods are an extension of $W$-methods we first construct order conditions for the $W$-methods, then prove two lemmas that allow us to obtain the specific exponential-$K$ order conditions.  We follow a similar derivation procedure to that outlined in \cite{Rainwater:2013}.  Throughout the derivation we  track the progress of several truncated B-series which include all terms up to order four. These series have 21 terms corresponding to the TW-trees shown in Figure \ref{fig:TWtrees}.

We begin the construction of order conditions for the $W$-method with a truncated $B$-series for $y_n$
\begin{equation*}
B^\#\left(y(t_0)\right) = \mathsf{a}_0 = \left\{ \mathsf{a}_0(\emptyset) = 1, x_i = 0 ~~\forall i = 1,\dots,21 \right\},
\end{equation*}
and then progress through individual stages of the $W$-method in equation (\ref{eqn:genformW}), making use of the formulas from Figure \ref{fig:TWtrees} to construct the resultant $B$-series of the composition and multiplication operations.  Algorithm \ref{alg:bseries} gives a method of constructing the $B$-series of the numerical solution $y_n$ that approximates the exact solution $y(t_0 + h)$. Note that the sum of two $B$-series is another $B$-series with coefficients equal to the sum of individual coefficients of the series being combined. Similarly, the product of a $B$-series with a scalar is a new series with each coefficient multiplied by the scalar. 
\begin{algorithm}[ht]
\caption{Construction of $B$-series of the numerical solution of a fourth-order, $s$ stage exponential-$W$-method (\ref{eqn:genformW})}
\label{alg:bseries}
\begin{algorithmic}[l]
  \For{ $i = 1, \dots, s$}
	\State $\mathsf{u} = \mathsf{a}_0$
	\For{ $ j = 1, \dots, i-1$}
		\State $\mathsf{u} = \mathsf{u} + \alpha_{i,j}\cdot \mathsf{k}_i$
	\EndFor
  	\State $\mathsf{q} = B^\#\left(hf\left(B(\mathsf{u},y)\right)\right) $
	\For{$ j = 1, \dots, i-1$}
		\State $\mathsf{q} = \mathsf{q} + \gamma_{i,j}\cdot B^\#\left(h\A_n B(\mathsf{k}_i,y)\right)$
	\EndFor
	\State $\mathsf{k}_i = B^\#\left(\varphi(h\gamma \A_n)B(\mathsf{q},y)\right) $
  \EndFor
  \State $\mathsf{a}_n = \mathsf{a}_0$
  \For{ $i = 1, \dots, s$ }
	\State $\mathsf{a}_n = \mathsf{a}_n + b_i\cdot\mathsf{k}_i$
  \EndFor
\end{algorithmic}
\end{algorithm}

The order conditions of the $W$-methods are obtained by matching the $B$-series coefficients of the exact solution $B^\#\left(y(t_n+h)\right)$ with those of the numerical solution $B^\#\left(\mathsf{y}_{n+1}\right)$ up to a specified order.  Keeping in mind that we do not ultimately seek order conditions for a $W$-method itself, that they are simply a means to an end, we look now at the process for obtaining order conditions of the $K$-method from this result.

The extension of the theory of $TW$-trees to $TK$-trees is done in \cite{Tranquilli:2014}.  This extension allows us to ``recolor'' all linear sub-trees (possessing only singly branched nodes) of the $TW$-trees and to substantially reduce the number of required conditions.  This is done using Lemmas \ref{lemma:K-matrix} and \ref{lemma:K-differentials}, taken from  \cite{Tranquilli:2014}, and repeated here without proof.

\begin{lemma}[Property of the Krylov approximate Jacobian \eqref{eqn:approxJac} \cite{Tranquilli:2014}]\label{lemma:K-matrix} 
For any $0 \le k \le M-1$ it holds that
\[
\mathbf{A}_n^k\, f_n = \mathbf{J}_n^k\, f_n\,,
\]
where $M=dim(\K_M)$.
\end{lemma}

\begin{lemma}[Property of elementary differentials using the approximation \eqref{eqn:approxJac}  \cite{Tranquilli:2014}]\label{lemma:K-differentials} 
When the Krylov approximation matrix \eqref{eqn:approxJac} is used in equation (\ref{eqn:genformW}), all {\it linear} TW-trees of order $k \le M$ correspond to a single elementary differential, regardless of the color of their nodes.  
\end{lemma}

$TK$-trees are the result of an application of Lemmas \ref{lemma:K-matrix} and \ref{lemma:K-differentials} to reduce the set of $TW$-trees that need to be considered
in the order conditions when the Krylov approximation matrix \eqref{eqn:approxJac} is used \cite{Tranquilli:2014}.  
\begin{definition}[TK-trees  \cite{Tranquilli:2014}]
\begin{eqnarray*}
TK &=& \left\{ TW\textrm{-trees:}  \textrm{ no linear sub-tree has a fat root}  \right\} \\ 
TK(k) &=& \left\{ TW\textrm{-trees:}  \textrm{ no linear sub-tree of order} \right.\\
                                 && \left. \qquad \qquad \textrm{ smaller than or equal to }   k  \textrm{ has a fat root}  \right\} \,.
\end{eqnarray*}
\end{definition}
Figure \ref{fig:TKtrees} shows all $TK$-trees to order four and the corresponding exponential-$K$ order conditions.  Note that there are only nine $TK$-trees as opposed to the original twenty $TW$-trees.  There is only one additional order condition compared to methods which make use of the exact Jacobian, and this condition corresponds to a tree which has a doubly-branched node occurring as a descendant of a fat node.

\begin{remark}
The order conditions given here are only for the exponential-$K$ methods, but the same process can be used to rederive the Rosenbrock-$K$ conditions given in \cite{Tranquilli:2014} through the use of different $c_n$ in the Taylor expansion of $\varphi(h\gamma\A_n)$ in equation (\ref{eqn:varphitaylor}).
\end{remark}

\begin{figure}[htp]
  \label{fig:TKtrees}
  \caption{TK-trees and exponential-$K$ conditions up to order four.}
  \begin{center}
  \begin{tabular}{|c|c|c|c|}
\hline
$\tau$ & $F(\tau)$ & $\Phi(\tau)$ & $P_\tau(\gamma)$ \\
\hline
  \begin{tikzpicture}[scale=.5]
      [meagre/.style={circle,draw, fill=black!100,thick},
      fat/.style={circle,draw,thick}]
      \node[circle,draw, fill=black!100,thick] (j) at (0,0) [label=right:$j$] {};
  \end{tikzpicture} & $ f^J $ & $1$ & $1$ \\
\hline
  \begin{tikzpicture}[scale=.5]
      [meagre/.style={circle,draw, fill=black!100,thick},
      fat/.style={circle,draw,thick}]
      \node[circle,draw, fill=black!100,thick] (j) at (0,0) [label=right:$j$] {};
      \node[circle,draw, fill=black!100,thick] (k) at (1,1) [label=right:$k$] {};
      \draw[-] (j) -- (k);
  \end{tikzpicture} & $ f^J_Kf^K $ & $\sum \beta_{j,k}$ & $1/2\, (1-\gamma)$ \\
\hline
  \begin{tikzpicture}[scale=.5]
      [meagre/.style={circle,draw, fill=black!100,thick},
      fat/.style={circle,draw,thick}]
      \node[circle,draw, fill=black!100,thick] (j) at (1,0) [label=right:$j$] {};
      \node[circle,draw, fill=black!100,thick] (k) at (2,1) [label=right:$k$] {};
      \node[circle,draw, fill=black!100,thick] (l) at (0,1) [label=right:$l$] {};
      \draw[-] (j) -- (k);
      \draw[-] (j) -- (l);
  \end{tikzpicture} & $ f^J_{KL}f^Kf^L $ & $\sum \alpha_{j,k}\alpha_{j,l}$  & $1/3$ \\
\hline
  \begin{tikzpicture}[scale=.5]
      [meagre/.style={circle,draw, fill=black!100,thick},
      fat/.style={circle,draw,thick}]
      \node[circle,draw, fill=black!100,thick] (j) at (0,0) [label=right:$j$] {};
      \node[circle,draw, fill=black!100,thick] (k) at (1,1) [label=right:$k$] {};
      \node[circle,draw, fill=black!100,thick] (l) at (0,2) [label=right:$l$] {};
      \draw[-] (j) -- (k);
      \draw[-] (k) -- (l);
  \end{tikzpicture} & $ f^J_Kf^K_Lf^L $ & $\sum \beta_{j,k}\beta_{k,l}$ & $1/3(1/2-\gamma)(1-\gamma)$ \\
\hline
  \begin{tikzpicture}[scale=.5]
      [meagre/.style={circle,draw, fill=black!100,thick},
      fat/.style={circle,draw,thick}]
      \node[circle,draw, fill=black!100,thick] (j) at (1,0) [label=right:$j$] {};
      \node[circle,draw, fill=black!100,thick] (k) at (2,1) [label=right:$k$] {};
      \node[circle,draw, fill=black!100,thick] (l) at (1,1) [label=above:$l$] {};
      \node[circle,draw, fill=black!100,thick] (m) at (0,1) [label=left:$m$] {};
      \draw[-] (j) -- (k);
      \draw[-] (j) -- (l);
      \draw[-] (j) -- (m);
  \end{tikzpicture} & $ f^J_{KLM}f^Kf^Lf^M $ & $\sum \alpha_{j,k}\alpha_{j,l}\alpha_{jm}$ & $1/4$ \\
\hline
  \begin{tikzpicture}[scale=.5]
      [meagre/.style={circle,draw, fill=black!100,thick},
      fat/.style={circle,draw,thick}]
      \node[circle,draw, fill=black!100,thick] (j) at (1,0) [label=right:$j$] {};
      \node[circle,draw, fill=black!100,thick] (k) at (2,1) [label=right:$k$] {};
      \node[circle,draw, fill=black!100,thick] (l) at (1,2) [label=right:$l$] {};
      \node[circle,draw, fill=black!100,thick] (m) at (0,1) [label=left:$m$] {};
      \draw[-] (j) -- (k);
      \draw[-] (k) -- (l);
      \draw[-] (j) -- (m);
  \end{tikzpicture} & $ f^J_{KM}f^K_Lf^Lf^M $ & $\sum \alpha_{j,k}\beta_{k,l}\alpha_{j,m}$ & $1/8-\gamma/6$ \\
\hline
  \begin{tikzpicture}[scale=.5]
      [meagre/.style={circle,draw, fill=black!100,thick},
      fat/.style={circle,draw,thick}]
      \node[circle,draw, fill=black!100,thick] (j) at (0,0) [label=right:$j$] {};
      \node[circle,draw, fill=black!100,thick] (k) at (1,1) [label=right:$k$] {};
      \node[circle,draw, fill=black!100,thick] (l) at (0,2) [label=above:$l$] {};
      \node[circle,draw, fill=black!100,thick] (m) at (2,2) [label=above:$m$] {};
      \draw[-] (j) -- (k);
      \draw[-] (k) -- (l);
      \draw[-] (k) -- (m);
  \end{tikzpicture} & $ f^J_Kf^K_{LM}f^Lf^M $ & $\sum \alpha_{j,k}\alpha_{k,m}\alpha_{k,l}$ & $1/12$ \\
\hline
  \begin{tikzpicture}[scale=.5]
      [meagre/.style={circle,draw, fill=black!100,thick},
      fat/.style={circle,draw,thick}]
      \node[circle,draw, thick] (j) at (0,0) [label=right:$j$] {};
      \node[circle,draw, fill=black!100,thick] (k) at (1,1) [label=right:$k$] {};
      \node[circle,draw, fill=black!100,thick] (l) at (0,2) [label=above:$l$] {};
      \node[circle,draw, fill=black!100,thick] (m) at (2,2) [label=above:$m$] {};
      \draw[-] (j) -- (k);
      \draw[-] (k) -- (l);
      \draw[-] (k) -- (m);
  \end{tikzpicture} & $ \mathbf{A}_{JK}f^K_{LM}f^Lf^M $ & $\sum \gamma_{j,k}\alpha_{k,m}\alpha_{k,l}$ & $-\gamma/6$ \\
\hline
  \begin{tikzpicture}[scale=.5]
      [meagre/.style={circle,draw, fill=black!100,thick},
      fat/.style={circle,draw,thick}]
      \node[circle,draw,fill=black!100, thick] (j) at (0,0) [label=left:$j$] {};
      \node[circle,draw, fill=black!100,thick] (k) at (1,1) [label=right:$k$] {};
      \node[circle,draw, fill=black!100,thick] (l) at (0,2) [label=left:$l$] {};
      \node[circle,draw, fill=black!100,thick] (m) at (1,3) [label=right:$m$] {};
      \draw[-] (j) -- (k);
      \draw[-] (k) -- (l);
      \draw[-] (l) -- (m);
  \end{tikzpicture} & $ f^J_Kf^K_Lf^L_Mf^M $ & $\sum \beta_{j,k}\beta_{k,l}\beta_{l,m}$ & $1/4(1/3-\gamma) (1/2-\gamma) (1-\gamma)$ \\
\hline
  \end{tabular}
  \end{center}
\end{figure}

\begin{theorem}[Order conditions for exponential-$K$ methods]\label{thm:K1-conditions}
An exponential-$K$ method has order $p$ iff the underlying Krylov space \eqref{eqn:krylovspace} has dimension $M \ge p$, and the following order conditions hold:
\label{eqn:ROK-conditions}
\begin{eqnarray}
\label{eqn:ROK-condition-T}
\sum_j b_j\, \Phi_j(\tau) = P_\tau(\tau) \quad \forall\, \tau \in TK ~~ \mbox{with } \left|\tau \right| \le p\,.
\end{eqnarray}
Here $\left|\tau \right|$ is the order, or number of vertices of the tree $\tau$, and $\Phi_j(\tau)$ and $P_\tau(\tau)$ are  computed using Algorithm \ref{alg:bseries}; 
they are shown in Figure \ref{fig:TKtrees} for $p \leq 4$.
\end{theorem}
\begin{proof}
The proof follows from our discussion, the near equivalence of order conditions for exponential-W and Rosenbrock-W methods \cite{Hochbruck:1998}, and from the order conditions of Rosenbrock-W methods \cite[Theorem 7.7]{Hairer_book_II}.
\qquad \qed
\end{proof}

\begin{remark}[Stiff order conditions.]
This section has developed classical order conditions that explain the accuracy of the methods on non-stiff problems. The behavior of the methods when applied to very stiff problems may be different, e.g., true to order reduction. A stiff order conditions theory for exponential methods has been proposed by Luan and Ostermann \cite{Luan:2014}. The development of stiff order conditions for exponential-K methods falls outside the scope of this paper.
\end{remark}

\begin{remark}[Stability considerations.]
The numerical stability of exponential-K solutions depends on the choice of Krylov space. Intuitively, the size of the K space should be large enough to cover the stiff subspace of the system.Note that traditional exponential methods focus on accurate computations of matrix-exp-vector products (e.g., by monitoring residuals), but do not account explicitly for the impact of Krylov approximations on stability. The large number of basis vectors required to achieve accurate matrix function vector products favors stability. In our case a small K dimension ensures accuracy, so the stability needs to be considered separately.  An automatic procedure to select dimension such as to achieve stability is important, but falls outside the scope here. 
\end{remark}

\section{An exponential-K method of order four}
\label{sec:newmethods}

We now construct an exponential-$K$ method of order four.  As before we consider the case where $\gamma_{i,i} = \gamma$ for all stages $i$ and denote
\begin{equation*}
 \beta_{i,j} = \alpha_{i,j} + \gamma_{i,j}, \quad \beta_i' = \displaystyle\sum_{j=1}^{i-1}\beta_{i,j}.
\end{equation*}
The following nine non-linear equations arise from the order conditions of a four stage, fourth order exponential-$K$ method 
\begin{equation}
\label{eqn:ocfours}
\renewcommand{\arraystretch}{1.5}
\begin{array}{clclcl}
	(a) & b_1 + b_2 + b_3 + b_4 & = & 1 & & \\
	(b) & b_2 \beta_2' + b_3 \beta_3' + b_4 \beta_4' & = & \frac{1}{2}(1 - \gamma) & = & p_{21}(\gamma)\\
	(c) & b_2 \alpha_2^2 + b_3 \alpha_3^2 + b_4 \alpha_4^2 & = &  \frac{1}{3} & & \\
	(d) & b_3(\beta_{3,2} \beta_2') + b_4(\beta_{4,2}\beta_2' + \beta_{4,3}\beta_3') & = & \frac{1}{3}(\frac{1}{2}-\gamma)(1-\gamma)& = & p_{3,2}(\gamma) \\
	(e) & b_2 \alpha_2^3 + b_3\alpha_3^3 + b_4\alpha_4^3 & = & \frac{1}{4} & & \\
	(f) & b_3\alpha_{3,2}\beta_2' + b_4(\alpha_{4,2}\beta_2' + \alpha_{4,3}\beta_3) & = & \frac{1}{8} - \frac{1}{6} \gamma & = & p_{4,2}(\gamma) \\
	(g_1) & b_3\alpha_{3,2}\alpha_2^2 + b_4(\alpha_{4,2}\alpha_2^2 + \alpha_{4,3}\alpha_3^2) & = & \frac{1}{12} & & \\
	(g_2) & b_3\gamma_{3,2}\alpha_2^2 + b_4(\gamma_{4,2}\alpha_2^2 + \gamma_{4,3}\alpha_3^2) & = & -\frac{1}{6}\gamma & &  \\
	(h) & b_4\beta_{4,3}\beta_{3,2}\beta_{2}' & = & \frac{1}{4}(\frac{1}{3}-\gamma)(\frac{1}{2}-\gamma)(1-\gamma)& = & p_{4,4}(\gamma)
\end{array}
\end{equation}
If we now set
\begin{equation*}
 p_{4,3}(\gamma) = \frac{1}{12}-\frac{1}{6}\gamma,
\end{equation*}
we can follow exactly the solution procedure given in \cite{Tranquilli:2014} for obtaining the {\sc rok}4a method, where we make use of the $p_{i,j}$ given above, and as suggested in \cite{Hochbruck:1998} to guarantee exact solutions for linear ODEs choose $\gamma$ as the reciprocal of an integer.  For the {\sc expK} method given in Table \ref{table:expK4coef} we make the arbitrary choices 
\begin{equation*}
 \gamma = \frac{1}{4}, \quad b_3 = 0, \quad \alpha_2 = 1, \quad \alpha_3 = \alpha_4 = \frac{1}{2}, \quad \beta_{4,3} = -\frac{1}{4}.
\end{equation*}
%
%
%
%
\begin{table}
\begin{center}
\renewcommand{\arraystretch}{1.5}
  \begin{tabular}{|rcrrcr|}
\hline
\multicolumn{6}{|l|}{   $\gamma = \frac{1}{4} $} \\
\hline
   $\alpha_{2,1}$ & $=$ & $1$ 			& $\gamma_{2,1}$ & $=$ & $\frac{7}{8}$ \\
   $\alpha_{3,1}$ & $=$ & $\frac{41}{80}$ 	& $\gamma_{3,1}$ & $=$ & $\frac{1}{16}$ \\
   $\alpha_{3,2}$ & $=$ & $\frac{1}{80}$ 	& $\gamma_{3,2}$ & $=$ & $0$ \\
   $\alpha_{4,1}$ & $=$ & $\frac{1}{4}$		& $\gamma_{4,1}$ & $=$ & $-\frac{1}{32}$ \\
   $\alpha_{4,2}$ & $=$ & $\frac{1}{12}$	& $\gamma_{4,2}$ & $=$ & $\frac{1}{24}$ \\
   $\alpha_{4,3}$ & $=$ & $\frac{1}{6}$ 	& $\gamma_{4,3}$ & $=$ & $-\frac{5}{12}$ \\
\hline
   $b_1$ 	 & $=$ & $\frac{1}{6}$ 	& $\widehat{b}_1$   & $=$ & $\frac{8}{3}$ \\
   $b_2$	 & $=$ & $\frac{1}{6}$ 	& $\widehat{b}_2$   & $=$ & $1$ \\
   $b_3$	 & $=$ & $0$		& $\widehat{b}_3$   & $=$ & $-\frac{8}{3}$ \\
   $b_4$	 & $=$ & $\frac{2}{3}$ 	& $\widehat{b}_4$   & $=$ & $0$		     \\
\hline
  \end{tabular}
\caption{Coefficients of {\sc expK}, a fourth order exponential-$K$ method.}
\label{table:expK4coef}
\end{center}
\end{table}
%

\section{Alternative implementations of existing exponential methods}
\label{sec:alternative}

We now consider alternative implementations  of previously derived methods {\sc exp4} \cite{Hochbruck:1998} and {\sc erow4} \cite{Hochbruck:2009}. These reformulations make use of only a single Krylov subspace projection per time step and exploit the B-series analysis of Section \ref{sec:orderconditions}.

The method  {\sc exp4} \cite{Hochbruck:1998} has the alternative formulation:
\begin{subequations}
\label{eqn:exp4}
\begin{eqnarray}
  && k_1 = \varphi_1\left(\frac{1}{3}h\A_n\right)f(y_n), \quad k_2 = \varphi_1(\frac{2}{3}h\A_n)f(y_n), \quad k_3 = \varphi_1(h\A_n)f(y_n),\\
  && w_4 = \frac{-7}{300}k_1 + \frac{97}{150}k_2 - \frac{37}{300}k_3, \\
  && u_4 = y_n + hw_4, \quad d_4 = f(u_4)-f(y_n) - h\A_n w_4, \\
  && k_4 = \varphi_1\left(\frac{1}{3}h\A_n\right)d_4, \quad k_5 = \varphi_1(\frac{2}{3}h\A_n)d_4, \quad k_6 = \varphi_1(h\A_n)d_4, \\
  && w_7 = \frac{59}{300}k_1 - \frac{7}{75}k_2 + \frac{269}{300}k_3 + \frac{2}{3}\left(k_4 + k_5 + k_6\right), \\
  && u_7 = y_n + hw_7, \quad d_7 = f(u_7) - f(y_n) - h\A_n w_7, \\
  && k_7 = \varphi_1\left(\frac{1}{3}h\A_n\right)d_7, \\
  && y_{1} = y_n + h\left(k_3 + k_4 - \frac{4}{3}k_5 + k_6 + \frac{1}{6}k_7\right).
\end{eqnarray}
\end{subequations}
The method {\sc erow4} \cite{Hochbruck:2009} has the alternative formulation
\begin{subequations}
\label{eqn:erow4}
\begin{eqnarray}
 && k_1 = \varphi_1(\frac{1}{2} h \A_n)f(y_n), \\
 && w_2 = \frac{1}{2}k_1, \\
 && u_2 = y_n + hw_2, \quad  d_2 = f(u_2) - f(y_n) - h \A_n w_2, \\
 && k_2 = \varphi_1(h\A_n)f(y_n), \quad k_3 = \varphi_1(h\A_n)d_2, \\
 && w_4 = k_2 + k_3, \\
 && u_4 = y_n + hw_4, \quad  d_4 = f(u_4) - f(y_n) - h \A_n w_4, \\
 && k_4 = \varphi_3(h\A_n)d_2, \quad k_5 = \varphi_4(h\A_n)d_2, \quad k_6 = \varphi_3(h\A_n)d_4, \quad k_7 = \varphi_4(h\A_n)d_4, \\
 && y_{n+1} = y_n + h\left(k_2 + 16k_4 - 48k_5 -2k_6 + 12k_7\right)
 \end{eqnarray}
\end{subequations}
We implement these methods in three different forms: first, in the standard way outlined in the literature \cite{Tokman_2011_EPIRK, Tokman_2006_EPI,Hochbruck:1998}; second, entirely in the reduced space such as given in \eqref{eqn:exp4}, \eqref{eqn:erow4}, and in \cite{Tranquilli:2014}; and finally, using only a single Krylov projection to approximate the $\varphi$ functions.  
These implementations are discussed below.

\subsection{Standard implementation}

The primary feature of a standard implementation of an exponential method is the approximation of $\varphi$ functions using Krylov subspaces.  For a term of the form $\varphi(h\A_n)b$ 
this is done by projecting $\varphi(h\A_n)$ and $b$ onto the space $\K_M(\A_n, b) = \textrm{span}\left\{b, \A_n b, \A_n^2 b, \dots, \A_n^{M-1} b\right\}$ as follows
\begin{equation}
\label{eqn:proj1}
 \varphi(h\A_n)b \approx \V\V^T \varphi(h\A_n) \V\V^Tb.
\end{equation}
Note that  $V^T b = \Vert b \Vert_2 e_1$, where $e_1$ is the first canonical basis vector.  
Making use of equation (\ref{eqn:arnH}) we obtain the final Krylov subspace approximation
\begin{equation}
 \label{eqn:standardprojection}
 \varphi(h\A_n)b \approx \Vert b \Vert_2 V\varphi(h\Hb)e_1\,.
\end{equation}
This approximation is computed as in \cite{Sidje:1998}, in which the exponential of an augmented matrix $\widetilde{\Hb}$ is constructed  \cite{Higham:2005} and (\ref{eqn:standardprojection}) is read off from this result.

\begin{remark}
 The standard implementation requires the construction of a new Krylov space for each vector $b$ operated on by a $\varphi$ function, as well as the evaluation of a small matrix exponential to compute each $\varphi(h \gamma \Hb)$ function.  Both {\sc exp}4 and {\sc erow}4 require the construction of three Krylov subspaces and the evaluation of seven small matrix exponentials.
\end{remark}

\subsection{K-type implementation}

K-type implementations of {\sc exp}4 and {\sc erow}4 follow the style of \cite{Tranquilli:2014} and Section 1.  For each $k_i \in \R^N$ we create a corresponding $\lambda_i = \V^Tk_i \in \R^M$, similarly $\sigma_i = \V^Tw_i \in \R^M$, and evaluate all linear algebra operations, including Jacobian-vector products, in the reduced space.  Further, we construct only a single Krylov subspace and perform full matrix computations of three $\varphi$ function evaluations in the case of {\sc exp}4, and four in the case of {\sc erow}4.

The $K$-type implementation of {\sc exp4}, called {\sc exp4k}, is:
\begin{subequations}
\label{eqn:exp4k}
\begin{eqnarray}
\renewcommand{\arraystretch}{1.5}
  \label{eqn:exp4k-a}
  && \nonumber \psi_0 = \V^T f(y_n), \quad f^{\perp}_0 = f - \V\psi_0, \\
  && \nonumber \lambda_1 = \varphi_1(\frac{1}{3}h\Hb)\psi_0, \quad \lambda_2 = \varphi_1(\frac{2}{3}h\Hb)\psi_0, \quad \lambda_3 = \varphi_1(h\Hb)\psi_0, \\
  && k_1 = \V\lambda_1 + f^{\perp}_0, \quad k_2 = \V\lambda_2 + f^{\perp}_0, \quad k_3 = \V\lambda_3 + f^{\perp}_0, \\
  && w_4 = \frac{-7}{300}k_1 + \frac{97}{150}k_2 - \frac{37}{300}k_3, \quad \sigma_4 = \frac{-7}{300}\lambda_1 + \frac{97}{150}\lambda_2 - \frac{37}{300}\lambda_3, \\
  && u_4 = y_n + hw_4, \quad \psi_4 = \V^Tf(u_4), \quad f^{\perp}_4 = f(u_4) - \V\psi_4, \quad \delta_4 = \psi_4 - \psi_0 - h\Hb\sigma_4,\\
  && \nonumber \lambda_4 = \varphi_1(\frac{1}{3} h \Hb)\delta_4, \quad \lambda_4 = \varphi_1(\frac{1}{3} h \Hb)\delta_4, \quad  \lambda_4 = \varphi_1(\frac{1}{3} h \Hb)\delta_4,\\
  && k_4 = \V\lambda_4 + f^{\perp}_4 - f^{\perp}_0, \quad k_5 = \V\lambda_5 + f^{\perp}_4 - f^{\perp}_0, \quad k_6 = \V\lambda_6 + f^{\perp}_4 - f^{\perp}_0, \\
  && \nonumber w_7 = \frac{59}{300}k_1 - \frac{7}{75}k_2 + \frac{269}{300}k_3 + \frac{2}{3}\left(k_4 + k_5 + k_6\right), \\   
  && \sigma_7 = \frac{59}{300}\lambda_1 - \frac{7}{75}\lambda_2 + \frac{269}{300}\lambda_3 + \frac{2}{3}\left(\lambda_4 + \lambda_5 + \lambda_6\right), \\
  && u_7 = y_n + hw_7, \quad \psi_7 = \V^Tf(u_7), \quad f^{\perp}_7 = f(u_7) - V\psi_7, \quad \delta_7 = \psi_7 - \psi_0 - h\Hb\sigma_7,\\
  && \lambda_7 = \varphi_1(\frac{1}{3}h\Hb)\delta_7, \quad k_7 = \V^T\lambda_7 + f^{\perp}_7 - f^{\perp}_0, \\
  && y_{1} = y_n + h\left(k_3 + k_4 - \frac{4}{3}k_5 + k_6 + \frac{1}{6}k_7\right).
\end{eqnarray}
\end{subequations}

The $K$-type implementation of {\sc erow4}, called {\sc erow4k}, is:
\begin{subequations}
\label{eqn:erow4k}
\begin{eqnarray}
  && \nonumber \psi_0 = \V^Tf(y_n), \quad f^{\perp}_0 = f(y_n) - \V\psi_0, \\
  && \lambda_1 = \varphi_1(\frac{1}{2}h\Hb)\psi_0, \quad k_1 = \V\lambda_1 + f^{\perp}_0, \\
  && w_2 = \frac{1}{2}k_1, \quad \sigma_2 = \frac{1}{2}\lambda_1, \\
  && u_2 = y_n + hw_2, \quad \psi_2 = \V^Tf(u_2), \quad f^{\perp}_2 = f(u_2) - \V\psi_2, \quad \delta_2 = \psi_2 - \psi_0 - h\Hb\sigma_2, \\
  && \nonumber \lambda_2 = \varphi_1(h\Hb)\psi_0, \quad k_2 = \V\lambda_2 + f^{\perp}_0, \\
  &&  \lambda_3 = \varphi_1(h\Hb)\delta_2, \quad k_3 = \V\lambda_3 + f^{\perp}_2 - f^{\perp}_0, \\
  && w_4 = k_2 + k_3, \quad \sigma_4 = \lambda_2 + \lambda_3, \\
  && u_4 = y_n + hw_4, \quad \psi_4 = \V^Tf(u_4), \quad f^{\perp}_4 = f(u_4) - V\psi_4, \quad \delta_4 = \psi_4 - \psi_0 - h\Hb\sigma_4, \\
  && \nonumber \lambda_4 = \varphi_3(h\Hb)\delta_2, \quad  k_4 = \V\lambda_4 + \frac{1}{3!}\left( f^{\perp}_2 - f^{\perp}_0\right), \\
  && \nonumber \lambda_5 = \varphi_4(h\Hb)\delta_2,, \quad k_5 = \V\lambda_5 + \frac{1}{4!}\left( f^{\perp}_2 - f^{\perp}_0\right), \\
  && \nonumber \lambda_6 = \varphi_3(h\Hb)\delta_4, \quad k_6 = \V\lambda_6 + \frac{1}{3!}\left(f^{\perp}_4 - f^{\perp}_0\right), \\
  && \lambda_7 = \varphi_4(h\Hb)\delta_4, \quad k_7 = \V\lambda_7 + \frac{1}{4!}\left(f^{\perp}_4 - f^{\perp}_0\right), \\
  && y_{n+1} = y_n + h\left(k_2 + 16k_4 - 48k_5 -2k_6 + 12k_7\right)
\end{eqnarray}
\end{subequations}

\subsection{Single projection implementation}

The results of Section \ref{sec:orderconditions} imply that a single Krylov subspace need to be computed per time step guarantee the order of accuracy.
In contradistinction the standard implementation constructs several Krylov spaces, primarily due to the use of residuals indicating how accurately the matrix function times vector products have been approximated.  In the single projection implementation we construct only the one subspace, and similarly to $K$-type implementation compute the full matrix result of the $\varphi$ functions.  The implementation differs from the $K$-type implementation in that the linear algebra operations, including Jacobian-vector products, are computed in the full space.

The single projection implementation of {\sc exp4}, called {\sc exp4sp}, is:
\begin{subequations}
\label{eqn:exp4sp}
\begin{eqnarray}
  && \nonumber \psi_0 = \V^Tf(y_n), \quad f^{\perp}_0 = f(y_n) - \V\psi_0, \\
  && \nonumber \lambda_1 = \varphi_1(\frac{1}{3}h\Hb)\psi_0, \quad \lambda_2 = \varphi_1(\frac{2}{3}h\Hb)\psi_0, \quad \lambda_3 = \varphi_1(h\Hb)\psi_0, \\
  && k_1 = \V\lambda_1 + f^{\perp}_0, \quad k_2 = \V\lambda_2 + f^{\perp}_0, \quad k_3 = \V\lambda_3 + f^{\perp}_0, \\
  && w_4 = \frac{-7}{300}k_1 + \frac{97}{150}k_2 - \frac{37}{300}k_3, \\
  && u_4 = y_n + hw_4, \quad d_4 = f(u_4) - f(y_n) - h\J_n w_4,\\
  && \nonumber \psi_4 = \V^Td_4, \quad d_4^{\perp} = d_4 - \V\psi_4, \\
  && \nonumber \lambda_4 = \varphi_1(\frac{1}{3}h\Hb)\psi_4, \quad \lambda_5 = \varphi_1(\frac{2}{3}h\Hb)\psi_4, \quad \lambda_6 = \varphi_1(h\Hb)\psi_4, \\
  && k_4 = \V\lambda_4 + d_4^{\perp}, \quad k_5 = \V\lambda_5 + d_4^{\perp}, \quad k_6 = \V\lambda_6 + d_4^{\perp}, \\
  && w_7 = \frac{59}{300}k_1 - \frac{7}{75}k_2 + \frac{269}{300}k_3 + \frac{2}{3}\left(k_4 + k_5 + k_6\right), \\
  && u_7 = y_n + hw_7, \quad d_7 = f(u_7) - f(y_n) - h\J_n w_7, \\ 
  && \nonumber \psi_7 = \V^Td_7, \quad d_7^{\perp} = d_7 - \V\psi_7, \\
  && \lambda_7 = \varphi_1(\frac{1}{3}h\Hb)\psi_7, \quad k_7 = \V\lambda_7 + d_7^{\perp}, \\
  && y_{1} = y_n + h\left(k_3 + k_4 - \frac{4}{3}k_5 + k_6 + \frac{1}{6}k_7\right).
\end{eqnarray}
\end{subequations}
The single projection implementation of {\sc erow4}, called {{\sc erow4sp}, is:
\begin{subequations}
\label{eqn:erow4sp}
\begin{eqnarray}
  && \nonumber \psi_0 = \V^Tf(y_n), \quad f_0^{\perp} = f(y_n) - \V\psi_0, \\
  && \lambda_1 = \varphi_1(\frac{1}{2}h\Hb)f(y_n), \quad k_1 = \V\lambda_1 + f_0^{\perp}, \\
  && w_2 = \frac{1}{2}k_1, \\
  && u_2 = y_n + hw_2, \quad d_2 = f(u_2) - f(y_n) - h\J_n w_2, \\
  && \nonumber \psi_2 = V^Td_2, \quad d_2^{\perp} = d_2 - \V\psi_2, \\
  && \nonumber \lambda_2 = \varphi_1(h\Hb)\psi_0, \quad k_2 = \V\lambda_2 + f_0^{\perp}, \\
  && \lambda_3 = \varphi_1(h\Hb)\psi_2, \quad k_3 = \V\lambda_3 + d_2^{\perp}, \\
  && w_4 = k_2 + k_3, \\
  && u_4 = y_n + hw_4, \quad d_4 = f(u_4) - f(y_n) - h\J_n w_4, \\
  && \nonumber \psi_4 = V^Td_4, \quad d_4^{\perp} = d_4 - \V\psi_4, \\
  && \nonumber \lambda_4 = \varphi_3(h\Hb)d_2, \quad k_4 = \V\lambda_4 + \frac{1}{3!}d_2^{\perp}, \quad \lambda_5 = \varphi_4(h\Hb)d_2, \quad k_5 = \V\lambda_5 + \frac{1}{4!}d_2^{\perp}, \\
  && \lambda_6 = \varphi_3(h\Hb)d_4, \quad k_6 = \V\lambda_6 + \frac{1}{3!}d_4^{\perp}, \quad \lambda_7 = \varphi_4(h\Hb)d_4, \quad k_7 = \V\lambda_7 + \frac{1}{4!}d_4^{\perp}, \\
  && y_{n+1} = y_n + h\left(k_2 + 16k_4 - 48k_5 -2k_6 + 12k_7\right).
\end{eqnarray}
\end{subequations}

\subsection{Accuracy analysis of alternative implementations}

Using the approach described by Algorithm \ref{alg:bseries} we construct B-series representations of the numerical solutions produced by {\sc exp4k}, {\sc exp4sp}, {\sc erow4k}, and {\sc erow4sp}.  Table \ref{table:Banalysis} shows the B-series coefficients for up to fourth order.  Note that coefficients associated to various trees change not only for different methods but also for different formulations of the same method.

The critical coefficient is that belonging to $\tau_{13}$, i.e., corresponding to the $K$ order condition.  Table \ref{table:Banalysis} reveals that {\sc exp4k} is fourth order, while both {\sc exp4sp}, {\sc erow4K}, and {\sc erow4sp} are only third order.  These analytical results are confirmed experimentally in the next section.

\begin{table}
\begin{center}
\renewcommand{\arraystretch}{1.5}
 \begin{tabular}{|c|c|c|c|c|c|c|}
  \hline
  $i$ & $F(\tau_i)$ & {\sc exp4k} &{\sc exp4sp} &{\sc erow4k} & {\sc erow4sp} & Exact Solution\\
  \hline
  $1$ 	& $f^J$ 			& $1$ & $1$ & $1$ & $1$ & $1$ \\
  \hline
  $2$ 	& $f^J_Kf^K$ 			& $\frac{1}{2}$& 0 & $\frac{1}{2}$ & 0 & \multirow{2}{*}{$\frac{1}{2}$} \\
  \cline{1-6}
  $3$ 	& $\A_{JK}f^K$ 			& 0 & $\frac{1}{2}$ & $0$ & $\frac{1}{2}$ & \\
  \hline
  $4$ 	& $f^J_{KL}f^Kf^L$ 		& $\frac{1}{3}$ & $\frac{1}{3}$ & $\frac{1}{3}$ & $\frac{1}{3}$ & $\frac{1}{3}$\\
  \hline
  $5$ 	& $f^J_K f^K_L f^L$ 		&$\frac{1}{6}$ & 0 & $\frac{1}{12}$ & 0 & \multirow{4}{*}{$\frac{1}{6}$}\\
  \cline{1-6}
  $6$ 	& $f^J_K\A_{KL}f^L$ 		&$\frac{1}{120}$ & 0 & $\frac{1}{12}$ & 0 &  \\
  \cline{1-6}
  $7$ 	& $\A_{JK}f^K_Lf^L$ 		&-$\frac{1}{36}$ & 0 & $\frac{1}{15}$ & 0 & \\
  \cline{1-6}
  $8$ 	& $\A_{JK}\A_{KL}f^L$		&$\frac{7}{360}$ & $\frac{1}{6}$ & $-\frac{1}{15}$ & $\frac{1}{6}$ &\\
  \hline
  $9$ 	& $f^J_{KLM}f^Kf^Lf^M$		&$\frac{1}{4}$ & $\frac{1}{4}$ & $\frac{1}{4}$ & $\frac{1}{4}$ &  $\frac{1}{4}$ \\
  \hline
  $10$ 	& $f^J_{KL}f^L_Mf^Mf^K$		&$\frac{1}{6}$ & 0 & $\frac{1}{12}$ & 0 & \multirow{2}{*}{$\frac{1}{8}$}\\
  \cline{1-6}
  $11$	& $f^J_{KL}\A_{LM}f^Mf^K$	&$-\frac{1}{24}$& $\frac{1}{8}$ & $\frac{1}{24}$ & $\frac{1}{8}$ & \\
  \hline
  $12$	& $f^J_Kf^K_{LM}f^Mf^L$		&$\frac{1}{12}$& 0 & $\frac{1}{24}$ & 0 & $\frac{1}{12}$\\
  \hline
  $13$ 	& $\A_{JK}f^K_{LM}f^Mf^L$	&0& $\frac{1}{12}$ & $\frac{1}{24}$ & $\frac{1}{12}$ & 0 \\
  \hline
  $14$ 	& $f^J_Kf^K_Lf^L_Mf^M$		&0& 0 & 0 & 0 & \multirow{8}{*}{$\frac{1}{24}$}\\
  \cline{1-6}
  $15$	& $f^J_Kf^K_L\A_{LM}f^M$	&$\frac{1}{20}$& 0 & $\frac{1}{48}$ & 0 & \\
  \cline{1-6}
  $16$	& $f^J_K\A_{KL}f^L_Mf^M$	&$\frac{1}{18}$& 0 & $\frac{1}{24}$ & 0 & \\
  \cline{1-6}
  $17$	& $f^J_K\A_{KL}\A_{LM}f^M$	&$-\frac{1537}{24300}$& 0 & $-\frac{1}{48}$ & 0 & \\
  \cline{1-6}
  $18$	& $\A_{JK}f^K_Lf^L_Mf^M$	&$\frac{1}{36}$& 0 & $\frac{1}{120}$ & 0 & \\
  \cline{1-6}
  $19$	& $\A_{JK}f^K_L\A_{LM}f^M$	&$-\frac{23}{720}$& 0 & $\frac{1}{80}$ & 0 & \\
  \cline{1-6}
  $20$	& $\A_{JK}\A_{KL}f^L_Mf^M$	&$-\frac{1}{27}$ & 0 & $-\frac{1}{60}$ & 0 & \\
  \cline{1-6}
  $21$	& $\A_{JK}\A_{KL}\A_{LM}f^M$	&$\frac{3943}{97200}$ & $\frac{1}{24}$ & $-\frac{1}{240}$ & $\frac{1}{24}$ &  \\
  \hline
 \end{tabular}
\end{center}
\caption{B-series expansion of the numerical solution for different exponential methods and implementations. } 
\label{table:Banalysis}
\end{table}

\section{Numerical Results}
\label{sec:results}

We perform numerical tests using a nonlinear ODE model, the two-dimensional shallow water equations, and the two-dimensional Allen-Cahn problem. While we have constructed both $K$- and $SP$- type implementations of {\sc exp4} and {\sc erow4}, we present performance comparisons of only the fourth order methods {\sc expK}, {\sc exp4}, {\sc exp4K}, and {\sc erow4}, since the third order methods perform the same amount of work as fourth order methods but yield lower accuracy.

\subsection{Lorenz-96 model}
The chaotic Lorenz-96 model \cite{Lorenz1996} has $N=40$ states, periodic boundary conditions, and is described by the following equations:
\begin{eqnarray}
\label{LorenzModel}
\frac{dy_j}{dt} &=& -y_{j-1}\; \left(y_{j-2}-y_{j+1}\right)-y_j + F \;,
\quad j = 1, \ldots, N~,\\
\nonumber
y_{-1} &=& y_{N-1}~, \quad  y_{0} = y_N ~, \quad  y_{N+1} = y_{1}~.
\end{eqnarray}
The forcing term is $F=8.0$, with $t \in [0, \ 0.3]$ (time units).
%
\begin{figure}[htp]
\centering
\includegraphics[width=5in]{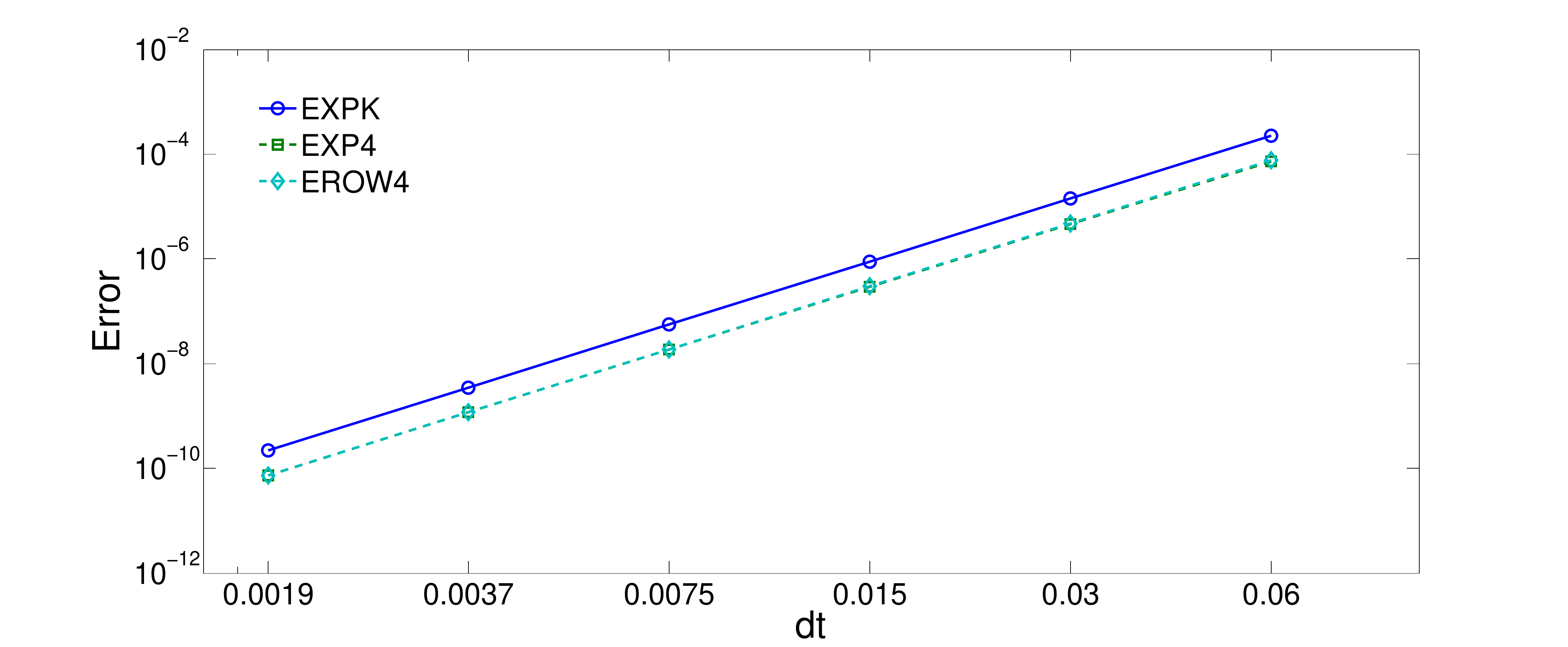} \\
\caption{Precision diagrams for different exponential methods in standard implementation applied to the Lorenz-96 test problem. {\sc expK} use a Krylov space of dimension $M = 5$.}
\label{fig:lorenzconvergence}
\end{figure}

%
\begin{table}
\begin{center}
\begin{tabular}{|c|c|c|c|}
\hline
		&	Standard 	&	K-type	&	SP-type \\
\hline
{\sc expK}	&	--		&	3.99	& 	--	\\
\hline
{\sc exp4}	&	3.98		&	3.97	&	2.97	\\
\hline
{\sc erow4}	&	4.00		&	2.97	&	2.96	\\
\hline
\end{tabular}
\caption{Convergence rates for all methods and implementations applied to the Lorenz-96 model.}
\label{table:lorenzconvergence}
\end{center}
\end{table}
%

Figure \ref{fig:lorenzconvergence} shows the precision diagrams for {\sc expK} with $M = 5$, and for the standard implementations of {\sc exp4} and {\sc erow4}.
All methods show the theoretical convergence order four.
The performance of different implementations of {\sc exp4} and {\sc erow4} are shown in Figures \ref{fig:lorenzconvergenceexp4} and \ref{fig:lorenzconvergenceerow4}, respectively.
The results confirm the lower orders of alternative implementations predicted by the B-series analysis presented in table \ref{table:Banalysis}.
The convergence rates for all methods applied to the Lorenz-96 model are summarized in Table \ref{table:lorenzconvergence} .


\begin{figure}[hp]
\centering
\subfigure[{\sc exp4}]{
\includegraphics[width=0.475\textwidth,height=0.3\textwidth]{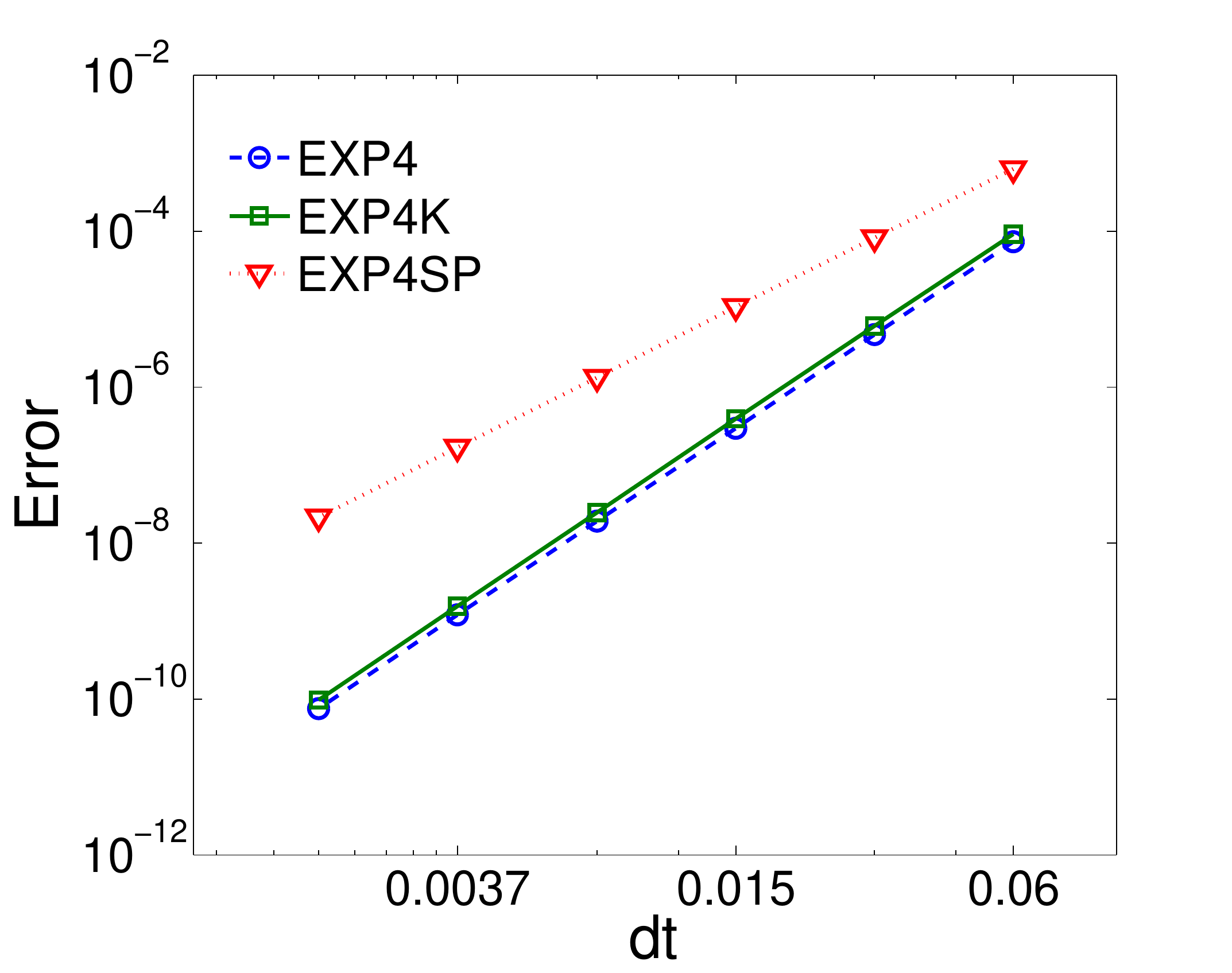}
\label{fig:lorenzconvergenceexp4}
}
\subfigure[{\sc erow4} ]{
\includegraphics[width=0.475\textwidth,height=0.3\textwidth]{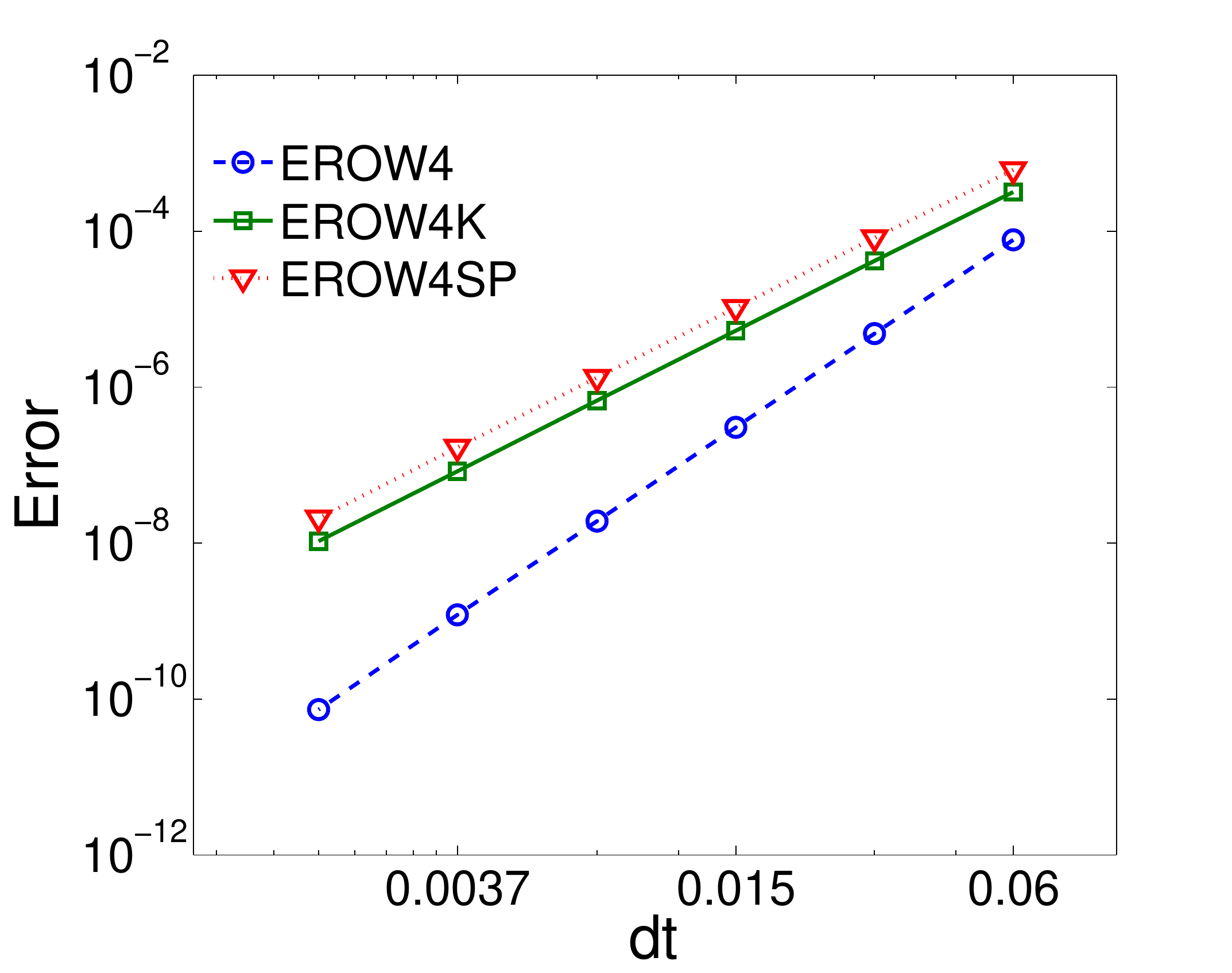}
\label{fig:lorenzconvergenceerow4}
}
\caption{Work-precision diagrams for different implementations of traditional exponential integrators applied to the  Lorenz-96 test problem \eqref{LorenzModel}.}
\label{fig:lorenzprecision}
\end{figure}

\subsection{Shallow water equations}

We examine the relative performance of the methods on the two-dimensional shallow water equations \cite{Liska97compositeschemes},
a hyperbolic system of partial differential equations
\begin{subequations}
\label{eqn:shallowwater}
\begin{eqnarray}
 \frac{\partial}{\partial t} h + \frac{\partial}{\partial x} (uh) + \frac{\partial}{\partial y} (vh) &=& 0, \\
 \frac{\partial}{\partial t} (uh) + \frac{\partial}{\partial x} \left(u^2 h + \frac{1}{2} g h^2\right) + \frac{\partial}{\partial y} (u v h) &=& 0,  \\
 \frac{\partial}{\partial t} (vh) + \frac{\partial}{\partial x} (u v h) + \frac{\partial}{\partial y} \left(v^2 h + \frac{1}{2} g h^2\right) &=& 0, 
\end{eqnarray}
\end{subequations}
where $u(x,y,t)$, $v(x,y,t)$ are the flow velocity components and $h(x,y,t)$ is the fluid height.  After spatial discretization using centered finite differences \eqref{eqn:shallowwater} is brought to the standard ODE form  (\ref{eqn:ode}) with
\begin{equation*}
y = \left[ u \, \, v \, \, h\right]^T \in \R^{N}, \quad f_y(t,y) = \mathbf{J} \in \R^{N \times N}.
\end{equation*}

The standard exponential integrators compute the product $\varphi(h\A_n)b$ with an adaptive basis size to guarantee accuracy and the comparisons include the cost of the extra residual computations required to do so, while the $K$- type implementations use a constant basis size chosen empirically for stability.  Automatic selection of Krylov basis size for stability is an open problem, and is the subject of future work. A special subroutine was implemented to compute exact Jacobian-vector products using a matrix-free approach.

\begin{figure}[hp]
\centering
\subfigure[$N = 3 \times 32 \times 32$]{
\includegraphics[width=0.475\textwidth,height=0.3\textwidth]{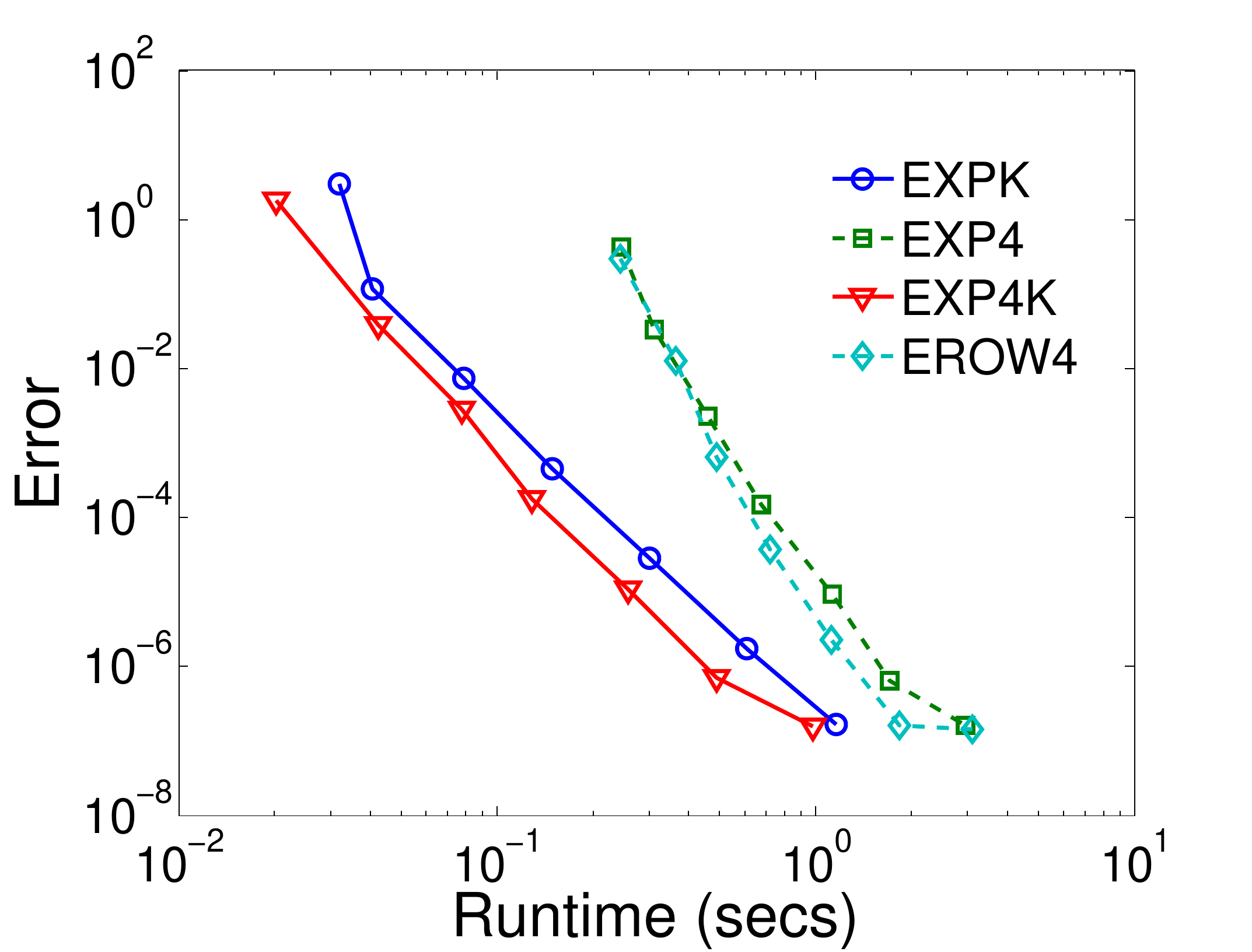}
\label{fig:sweprecision32}
}
\subfigure[$N = 3 \times 128 \times 128$]{
\includegraphics[width=0.475\textwidth,height=0.3\textwidth]{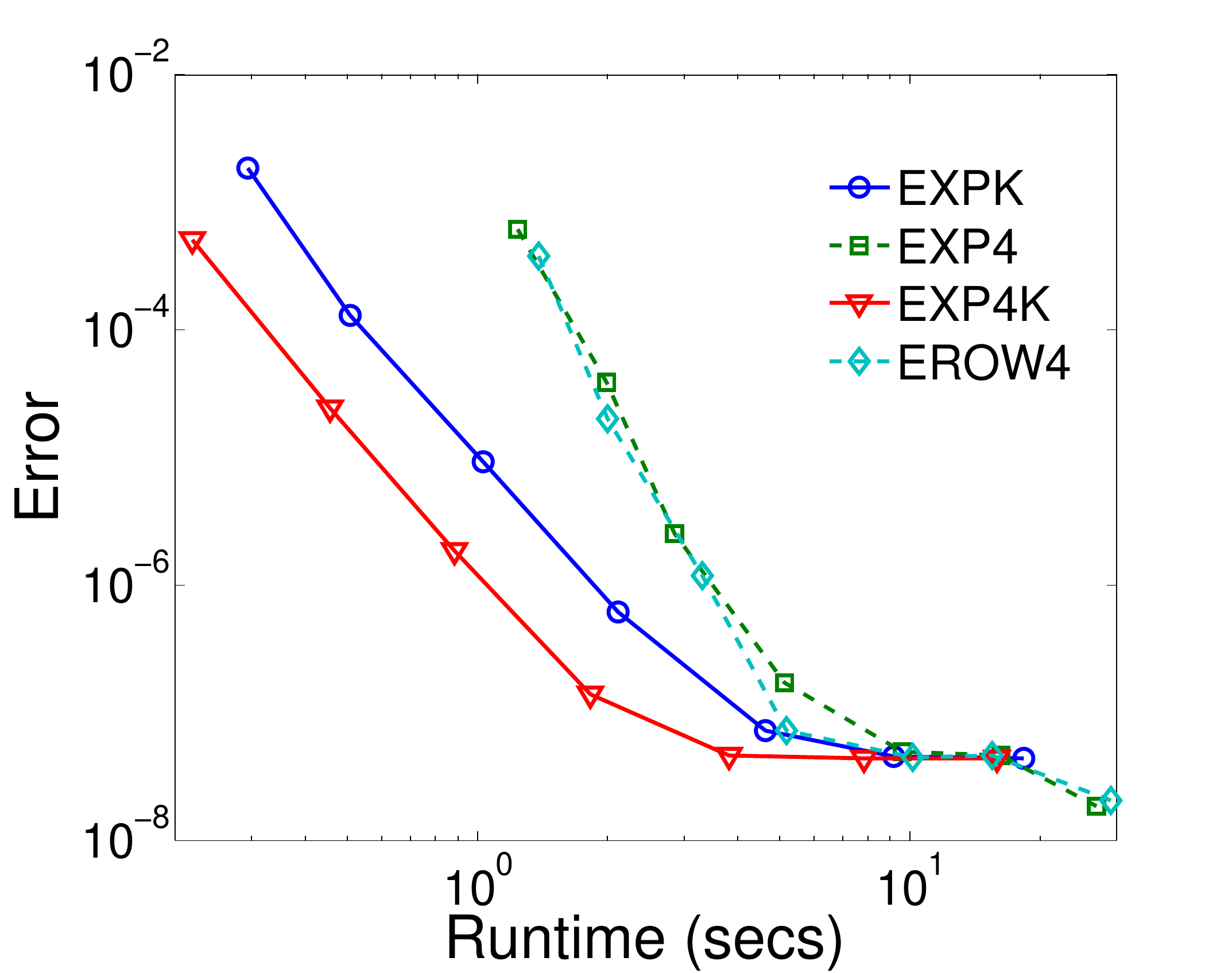}
\label{fig:sweprecision128}
}
\caption{Work-precision diagrams for exponential integrators applied to the shallow water test problem \eqref{eqn:shallowwater}.
Different problem sizes results from different spatial resolutions.}
\label{fig:sweprecision}
\end{figure}


Figures \ref{fig:sweprecision32} and \ref{fig:sweprecision128} show a performance comparison of the standard and $K$-type implementations of {\sc expK}, {\sc exp4}, and {\sc erow4}. 
Two grid sizes of $32 \times 32$ and $128 \times 128$ points are considered.  In both cases the $K$-type implementations are more efficient for lower error tolerances, while the adaptivity of the standard implementations allows them to `catch up' in performance as the errors decrease.


\subsection{The Allen-Cahn problem}

For further performance comparison we consider the two-dimensional Allen-Cahn system, a parabolic partial differential equation 
\begin{eqnarray}
\label{eqn:allencahn}
&& \frac{\partial}{\partial t} u = \alpha\nabla^2 u + \gamma\left(u - u^3\right), \quad (x,y) \in [0,1] \times [0,1], \quad  t \in [0, 0.2], 
\end{eqnarray}
with $\alpha = 0.1$ and $\gamma = 1.0$. The problem has homogeneous Neumann boundary conditions and the initial solution 
$u(t=0) = 0.4 + 0.1(x + y) + 0.1\sin(10x)\sin(20y)$.
Unlike the shallow water equations, the reaction-diffusion problem \eqref{eqn:allencahn} is stiff.

The standard implementations of various methods make use of an adaptive Krylov basis size while the $K$-type implementations use empirically selected basis sizes.  Figures \ref{fig:acprecision50} and \ref{fig:acprecision150} compares the performance of the standard and $K$-type implementations for different problem sizes.  We once again see a better efficiency of the $K$-type methods for lower error values, but with a much earlier break-even point for efficiency.  
This is due primarily to the spectrum of the diffusion operator, and the difference in stability and accuracy requirements between the shallow water and Allen-Cahn equations.  In the case of Allen-Cahn the stability requirements are more strict than the accuracy considerations and so the multiple smaller projections of the standard implementation become more efficient than the single larger projection in the $K$-type method, even though the latter uses fewer overall basis vectors.

\begin{figure}[hp]
\centering
\subfigure[$N = 50\times 50$]{
\includegraphics[width=0.475\textwidth,height=0.3\textwidth]{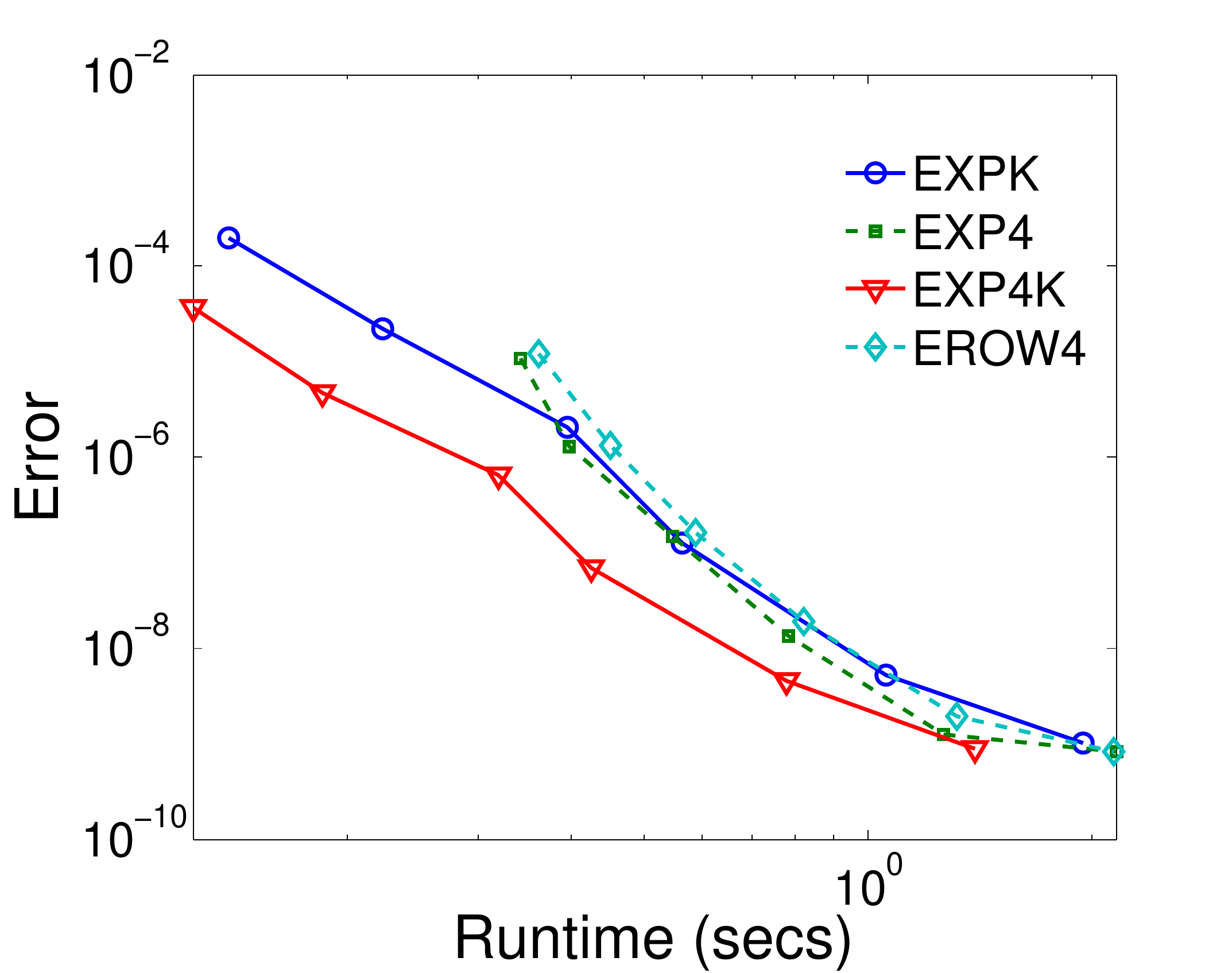}
\label{fig:acprecision50}
}
\subfigure[$N = 150\times 150$]{
\includegraphics[width=0.475\textwidth,height=0.3\textwidth]{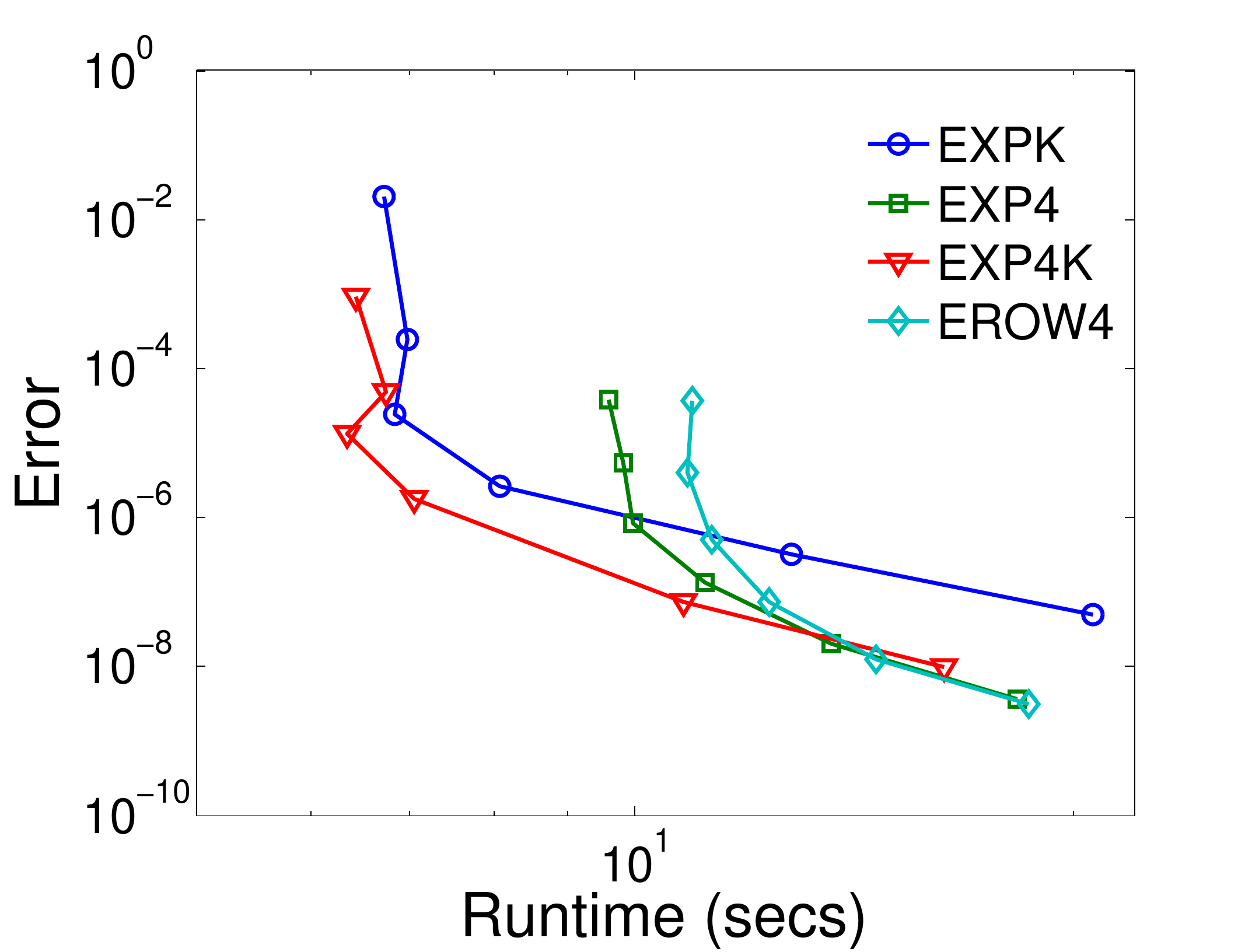}
\label{fig:acprecision150}
}
\caption{Work-precision diagrams for exponential integrators applied to the Allen-Cahn test problem \eqref{eqn:allencahn}.
Different problem sizes result from different spatial resolutions. }
\label{fig:acprecision}
\end{figure}

%

\section{Conclusions}\label{sec:conclusions}

This work extends the $K$-method approach proposed in \cite{Tranquilli:2014} to exponential integrators and develops
the new family of exponential-$K$ schemes. A rigorous framework for order conditions analysis is developed that accounts for both temporal truncation errors and
Krylov approximation errors. We construct an exponential-K method based on the general form given in \cite{Hochbruck:1998},
and reformulate existing exponential methods in order to take advantage of the reduced workload permitted by the new analysis.

Numerical experiments are carried out with three test problems, an ordinary differential equation and hyperbolic and parabolic partial differential equations. The results indicate that the new $K$-type exponential methods have the potential to be more efficient than their classical counterparts.
While the new $K$- method {\sc expK} derived here does not appear to be more efficient than the previously existing methods, primarily due to a less efficient general form, it validates the order conditions theory of exponential-$K$ methods.  We have shown that the traditional {\sc exp4} method satisfies the additional order four $K$-condition when reformulated as a $K$-method, and that the resulting {\sc exp4k} scheme is more efficient than previous methods for the test problems presented here.

Future work will focus on developing a methodology to automatically select the Krylov subspace size in order to guarantee numerical stability, as on the construction of new exponential methods that can take full advantage of the inherent benefits present in the $K$-type formulation.

\section*{Acknowledgements}

This work has been supported in part by NSF through awards NSF CMMI--1130667, 
NSF CCF--1218454, NSF CCF--0916493, 
 AFOSR FA9550--12--1--0293--DEF, AFOSR 12-2640-06,
and by the Computational Science Laboratory at Virginia Tech.

\appendix
\section{Order conditions for exponential-W methods.}

\begin{theorem}[Order conditions for exponential-$W$ methods]\label{thm:W-conditions}
An exponential-$W$ method with general form (\ref{eqn:genformW}) has order $p$ iff the following order conditions hold:
\label{eqn:expw-conditions}
\begin{eqnarray}
\label{eqn:expw-condition-T}
\sum_j b_j\, \Phi_j(\tau) = P_\tau(\tau) \quad \forall\, \tau \in TW ~~ \mbox{with } \left|\tau \right| \le p\,.
\end{eqnarray}
Here $\left|\tau \right|$ is the order, or number of vertices of the tree $\tau$, and $\Phi_j(\tau)$ and $P_\tau(\tau)$ can be computed using Algorithm \ref{alg:bseries}; 
they are shown in Figure \ref{table:expwcond} for $p \leq 4$.
\end{theorem}
\begin{proof}
The proof follows from our discussion in section \ref{sec:orderconditions}, the near equivalence of order conditions for exponential-W and Rosenbrock-W methods \cite{Hochbruck:1998}, and from the order conditions of Rosenbrock-W methods \cite[Theorem 7.7]{Hairer_book_II}.
\qquad \qed
\end{proof}

\begin{table}[ht]
\begin{center}
\renewcommand{\arraystretch}{1.5}
 \begin{tabular}{|c|c|c|c|}
  \hline
  $i$ & $F(\tau_i)$ 			& $\Phi(\tau)$ & $P_\tau(\gamma)$ \\
  \hline
  $1$ 	& $f^J$ 			& $1$ & $1$ \\
  \hline
  $2$ 	& $f^J_Kf^K$ 			& $\sum \alpha_{jk}$ & $\frac{1}{2}$ \\
  \hline
  $3$ 	& $\A_{JK}f^K$ 			& $\sum \gamma_{jk}$ & $\frac{-\gamma}{2}$ \\
  \hline
  $4$ 	& $f^J_{KL}f^Kf^L$ 		& $\sum \alpha_{jk}\alpha_{jl}$ & $\frac{1}{3}$ \\
  \hline
  $5$ 	& $f^J_K f^K_L f^L$ 		& $\sum \alpha_{jk}\alpha_{kl}$ & $\frac{1}{6}$ \\
  \hline
  $6$ 	& $f^J_K\A_{KL}f^L$ 		& $\sum \alpha_{jk}\gamma_{kl}$ & $\frac{-\gamma}{4}$ \\
  \hline
  $7$ 	& $\A_{JK}f^K_Lf^L$ 		& $\sum \gamma_{jk}\alpha_{kl}$ & $\frac{-\gamma}{4}$ \\
  \hline
  $8$ 	& $\A_{JK}\A_{KL}f^L$		& $\sum \gamma_{jk}\gamma_{kl}$ & $\frac{\gamma^2}{3}$ \\
  \hline
  $9$ 	& $f^J_{KLM}f^Kf^Lf^M$		& $\sum \alpha_{jk}\alpha_{jl}\alpha_{jm}$ & $\frac{1}{4}$  \\
  \hline
  $10$ 	& $f^J_{KL}f^L_Mf^Mf^K$		& $\sum \alpha_{jk}\alpha_{jl}\alpha_{jm}$ & $\frac{1}{8}$ \\
  \hline
  $11$	& $f^J_{KL}\A_{LM}f^Mf^K$	& $\sum \alpha_{jk}\alpha_{jl}\gamma_{lm}$ & $\frac{-\gamma}{6}$ \\
  \hline
  $12$	& $f^J_Kf^K_{LM}f^Mf^L$		& $\sum \alpha_{jk}\alpha_{kl}\alpha_{km}$ & $\frac{1}{12}$  \\
  \hline
  $13$ 	& $\A_{JK}f^K_{LM}f^Mf^L$	& $\sum \gamma_{jk}\alpha_{kl}\alpha_{km}$ & $\frac{-\gamma}{6}$  \\
  \hline
  $14$ 	& $f^J_Kf^K_Lf^L_Mf^M$		& $\sum \alpha_{jk}\alpha_{kl}\alpha_{lm}$ & $\frac{1}{24}$  \\
  \hline
  $15$	& $f^J_Kf^K_LA_{LM}f^M$		& $\sum \alpha_{jk}\alpha_{kl}\gamma_{lm}$ & $\frac{-\gamma}{12}$ \\
  \hline
  $16$	& $f^J_K\A_{KL}f^L_Mf^M$	& $\sum \alpha_{jk}\gamma_{kl}\alpha_{lm}$ & $\frac{-\gamma}{12}$  \\
  \hline
  $17$	& $f^J_K\A_{KL}\A_{LM}f^M$	& $\sum \alpha_{jk}\gamma_{kl}\gamma_{lm}$ & $\frac{\gamma^2}{6}$  \\
  \hline
  $18$	& $\A_{JK}f^K_Lf^L_Mf^M$	& $\sum \gamma_{jk}\alpha_{kl}\alpha_{lm}$ & $\frac{-\gamma}{12}$ \\
  \hline
  $19$	& $\A_{JK}f^K_L\A_{LM}f^M$	& $\sum \gamma_{jk}\alpha_{kl}\gamma_{lm}$ & $\frac{\gamma^2}{8}$  \\
  \hline
  $20$	& $\A_{JK}\A_{KL}f^L_Mf^M$	& $\sum \gamma_{jk}\gamma_{kl}\alpha_{lm}$ & $\frac{\gamma^2}{6}$  \\
  \hline
  $21$	& $\A_{JK}\A_{KL}\A_{LM}f^M$	& $\sum \gamma_{jk}\gamma_{kl}\gamma_{lm}$ & $\frac{-\gamma^3}{4}$ \\
  \hline
 \end{tabular}
\end{center}
\caption{Order conditions for exponential-W methods with general form (\ref{eqn:genformW}) and $p \leq 4$. } 
\label{table:expwcond}
\end{table}

\bibliographystyle{plain}
\bibliography{Master}

\begin{thebibliography}{10}

\bibitem{Butcher_2009_Bseries}
J.~C. Butcher.
\newblock Trees, {B}-series and exponential integrators.
\newblock {\em IMA Journal of Numerical Analysis}, 30(1):131--140, 2010.

\bibitem{Chartier_2010_Bseries}
Philippe Chartier, Ernst Hairer, and Gilles Vilmart.
\newblock Algebraic structures of {B}-series.
\newblock {\em Found. Comput. Math.}, 10(4):407--427, August 2010.

\bibitem{Hairer_book_I}
E.~Hairer, S.P. N{\o}rsett, and G.~Wanner.
\newblock {\em Solving Ordinary Differential Equations {II}: Stiff and
  Differential-Algebraic Problems}.
\newblock Springer, 2000.

\bibitem{Hairer_book_II}
E.~Hairer and G.~Wanner.
\newblock {\em Solving Ordinary Differential Equations {II}: Stiff and
  Differential-Algebraic Problems}.
\newblock Springer, 2002.

\bibitem{Higham:2005}
N.~Higham.
\newblock The scaling and squaring method for the matrix exponential revisited.
\newblock {\em SIAM Journal on Matrix Analysis and Applications},
  26(4):1179--1193, 2005.

\bibitem{Hochbruck:1998}
Marlis Hochbruck, Christian Lubich, and Hubert Selhofer.
\newblock Exponential integrators for large systems of differential equations.
\newblock {\em SIAM J. Sci. Comput.}, 19(5):1552--1574, September 1998.

\bibitem{Hochbruck:2009}
Marlis Hochbruck, Alexander Ostermann, and Julia Schweitzer.
\newblock Exponential {R}osenbrock-type methods.
\newblock {\em SIAM J. Numer. Anal.}, 47(1):786--803, February 2009.

\bibitem{Kelley_2004_Krylov}
C.T. Kelley, I.G. Kevrekidis, and L.~Qiao.
\newblock Newton-{K}rylov solvers for time-steppers.
\newblock arXiv:math/0404374v1, 2004.

\bibitem{Keyes_2004_JFNK}
D.A. Knoll and D.E. Keyes.
\newblock Jacobian-free {N}ewton--{K}rylov methods: a survey of approaches and
  applications.
\newblock {\em Journal of Computational Physics}, 193:357--397, 2004.

\bibitem{Liska97compositeschemes}
Richard Liska and Burton Wendroff.
\newblock Composite schemes for conservation laws, 1997.

\bibitem{Loffeld_2013_Comparison}
J.~Loffeld and M.~Tokman.
\newblock Comparative performance of exponential, implicit, and explicit
  integrators for stiff systems of {ODEs}.
\newblock {\em Journal of Computational and Applied Mathematics}, 241(0):45 --
  67, 2013.

\bibitem{Lorenz1996}
Edward~N.. Lorenz.
\newblock Predictability -- a problem partly solved.
\newblock In {\em Predictability of Weather and Climate}. Cambridge University
  Press, 2006.

\bibitem{Luan:2014}
Vu~Thai Luan and Alexander Ostermann.
\newblock Explicit exponential {R}unge – {K}utta methods of high order for
  parabolic problems.
\newblock {\em Journal of Computational and Applied Mathematics}, 256(0):168 --
  179, 2014.

\bibitem{Novati_2008_secantROW}
P.~Novati.
\newblock {Some secant approximations for Rosenbrock W-methods}.
\newblock {\em Applied Numerical Mathematics}, 58(3):195--211, 2008.

\bibitem{Podhaisky_1997_krylovW}
H.~Podhaisky, R.~Weiner, and B.A. Schmitt.
\newblock Numerical experiments with {K}rylov integrators.
\newblock {\em Applied Numerical Mathematics}, 28:413--425, 1997.

\bibitem{Rainwater:2013}
G.~Rainwater and M.~Tokman.
\newblock A new class of split exponential propagation iterative methods of
  runge–kutta type ({sEPIRK}) for semilinear systems of {ODEs}.
\newblock {\em Journal of Computational Physics}, 269(0):40 -- 60, 2014.

\bibitem{Rang_2005_ROW3}
J.~Rang and L.~Angermann.
\newblock New {R}osenbrock {W}-methods of order 3 for partial differential
  algebraic equations of index 1.
\newblock {\em BIT Numerical Mathematics}, 45(4):761--787, 2005.

\bibitem{Weiner_1997_rowmap}
R.Weiner, B.A. Schmitt, and H.~Podhaisky.
\newblock {ROWMAP}--a {ROW}-code with {K}rylov techniques for large stiff
  {ODE}s.
\newblock {\em Applied Numerical Mathematics}, 25:303--319, 1997.

\bibitem{Saad}
Y.~Saad.
\newblock {\em Iterative methods for sparse linear systems}.
\newblock PWS Pub. Co., Boston, 1996.

\bibitem{Schmitt_1995_krylovW}
B.A. Schmitt and R.~Weiner.
\newblock {Matrix free W-methods using a multiple Arnoldi iteration}.
\newblock {\em Applied Numerical Mathematics}, 18:307--320, 1995.

\bibitem{Sidje:1998}
Roger~B. Sidje.
\newblock Expokit: a software package for computing matrix exponentials.
\newblock {\em ACM Trans. Math. Softw.}, 24(1):130--156, March 1998.

\bibitem{Tokman_2006_EPI}
M.~Tokman.
\newblock Efficient integration of large stiff systems of {ODEs} with
  exponential propagation iterative ({EPI}) methods.
\newblock {\em Journal of Computational Physics}, 213(2):748 -- 776, 2006.

\bibitem{Tokman_2011_EPIRK}
M.~Tokman.
\newblock A new class of exponential propagation iterative methods of
  {R}unge-{K}utta type ({EPIRK}).
\newblock {\em J. Comput. Phys.}, 230(24):8762--8778, October 2011.

\bibitem{Tranquilli:2014}
Paul Tranquilli and Adrian Sandu.
\newblock Rosenbrock-krylov methods for large systems of differential
  equations.
\newblock {\em SIAM J. Scientific Computing}, 36(3), 2014.

\bibitem{vanderVorst}
Henk~A. van~der Vorst.
\newblock {\em Iterative Methods for Large Linear Systems}.
\newblock Cambridge University Press, 2003.

\bibitem{Weiner_1998_order}
R.~Weiner and B.A. Schmitt.
\newblock Order results for {Krylov-W} methods.
\newblock {\em Computing}, 61:69--89, 1998.

\end{thebibliography}

\end{document}